\documentclass[10pt]{amsart}
\usepackage{amsmath}
\usepackage{amsthm,amsxtra,amssymb,amsfonts,latexsym, mathrsfs, dsfont, bbm, bm, verbatim, mathtools}

\usepackage[initials,alphabetic]{amsrefs}
\usepackage[usenames]{color}
\usepackage[all]{xy}
\usepackage{hyperref}
\usepackage{mathrsfs}
\renewcommand{\vec}[1]{\underline{#1}}

\allowdisplaybreaks[1]

\newtheorem{lemma}{Lemma}[section]
\newtheorem{prop}[lemma]{Proposition}
\newtheorem{theorem}[lemma]{Theorem}
\newtheorem{cor}[lemma]{Corollary}

\theoremstyle{remark}
\newtheorem{rem}[lemma]{Remark}

\newtheorem{question}[lemma]{Question}

\theoremstyle{definition}
\newtheorem{defn}[lemma]{Definition}

\newcommand{\lge}{\langle}
\newcommand{\rge}{\rangle}

\newcommand{\ax}{\mathcal{A}}
\newcommand{\bx}{\mathcal{B}}

\newcommand{\fx}{\mathcal{F}}

\newcommand{\mx}{\mathcal{M}}

\newcommand{\nz}{\mathbb{N}}
\newcommand{\rz}{\mathbb{R}}
\newcommand{\cz}{\mathbb{C}}

\newcommand{\fz}{\mathbb{F}}

\newcommand{\Ga}{\Gamma}
\newcommand{\Om}{\Omega}

\newcommand{\ds}{\mathscr{D}}
\newcommand{\ps}{\mathscr{P}}
\newcommand{\js}{\mathscr{J}}
\newcommand{\ns}{\mathscr{N}}
\renewcommand{\ss}{\mathscr{S}}

\newcommand{\de}{\delta}
\newcommand{\si}{\sigma}

\newcommand{\8}{\infty}

\mathchardef\dash="2D

\newcommand{\Dom}{\operatorname{Dom}}

\newcommand{\Tr}{\operatorname{Tr}}

\begin{document}
\title{Free monotone transport for infinite variables}
\date{}
\author[Brent Nelson]{Brent Nelson$^\bullet$}
\address{$\bullet$ Department of Mathematics, University of California, Berkeley, CA 94709}
\email{brent@math.berkeley.edu}
\thanks{$\bullet$ Research supported by the NSF award DMS-1502822.}
\author[Qiang Zeng]{Qiang Zeng$^\circ$}
\address{$\circ$ Mathematics Department, Northwestern University,
2033 Sheridan Road Evanston, IL 60208}
\email{qzeng.math@gmail.com}
\thanks{$\circ$ Research supported by the NSF under Grant No. DMS-1440140 while the author was in residence at the MSRI in Berkeley, CA, during the Fall 2015 semester.}
\maketitle

\begin{abstract}
We extend the free monotone transport theorem of Guionnet and Shlyakhtenko to the case of infinite variables. As a first application, we provide a criterion for when mixed $q$-Gaussian algebras are isomorphic to $L(\mathbb{F}_\infty)$; namely, when the structure array $Q$ of a mixed $q$-Gaussian algebra has uniformly small entries that decay sufficiently rapidly. Here a mixed $q$-Gaussian algebra with structure array $Q=(q_{ij})_{i,j\in\mathbb{N}}$ is the von Neumann algebra generated by $X_n^Q=l_n+l_n^*, n\in\mathbb{N}$ and $(l_n)$ are the Fock space representations of the commutation relation $l_i^*l_j-q_{ij}l_jl_i^*=\delta_{i=j}, i,j\in\mathbb{N}$, $-1<q_{ij}=q_{ji}<1$.
\end{abstract}

\section{Introduction}
Classification is a central and difficult problem in the theory of operator algebras. In particular, the von Neumann algebras associated to free groups have been studied extensively ever since the early development of this field, while some major questions are still left open. It is well known that these von Neumann algebras, also known as the free group factors, are II$_1$ factors: they have trivial centers and finite traces. Due to the work of Voiculescu \cite{VDN}, it is known that the free group factors are isomorphic to the von Neumann algebras generated by free  semicircular variables. After Connes' deep theorem \cite{Con} that the von Neumann algebras of countable amenable groups with infinite conjugacy classes are all isomorphic to the hyperfinite II$_1$ factor, it is somewhat expected that the free group factors are the next important isomorphism class of II$_1$ factors; see Voiculescu's paper \cite{Voi06}.

On the other hand, motivated by quantum physics, Bo\.zejko and Speicher introduced a family of $C^*$-algebras and von Neumann algebras -- the $q$-Gaussian algebras \cite{BS91} and, more generally, the algebras associated to the Yang--Baxter equation \cite{BS94}. These algebras can be regarded as natural deformations of the operator algebras generated from free semicircular variables which correspond to the case $q=0$. We will call these deformed algebras the Bo\.zejko--Speicher algebras. This family of algebras contains both tracial and non-tracial von Neumann algebras, but in this paper we consider only the tracial case. A large body of work studying these algebras has been done in the last two decades. For an incomplete list, we mention \cites{BKS, Kr00, Sh04, Nou,Sn04,Kr05, Ri05,KN11,Av, JZ15}. All of these works show that the Bo\.zejko--Speicher algebras have similar properties to the operator algebras generated by free semicircular variables. Consequently, a natural question to ask is the following:

\begin{question}\label{ques0}
Are the Bo\.zejko--Speicher algebras (both $C^*$-algebras and von Neumann algebras) actually isomorphic when they have the same number of generators?
\end{question}
In fact, operator algebraists have become interested in this problem ever since the Bo\.zejko--Speicher algebras were introduced some 20 years ago. The situation when they have different numbers of generators in the case of von Neumann algebras is a generalized version of the famous free group factor isomorphism problem.

The first affirmative answer to Question \ref{ques0} is due to the recent works of Dabrowski, Guionnet and Shlyakhtenko in the case of finitely many generators. More precisely, Guionnet and Shlyakhtenko developed a remarkable theory of free monotone transport for a finite number of variables in \cite{GS14}, which can be regarded as the free analogue of Brenier's monotone transport theorem for Gaussian measures on $\rz^N$ \cite{Bre91}. Together with Dabrowski's work \cite{Dab} on the existence of conjugate variables and potentials for $q$-Gaussian variables, they showed that the $q$-Gaussian von Neumann algebras with $N$ generators are all isomorphic to that generated from $N$ free semicircular variables, provided $|q|$ is small enough. Thus they are isomorphic to the free group factor $L(\fz_N)$ associated to the free group $\fz_N$ with $N$ generators. Here the bound on $|q|$ depends on $N$, and degenerates to $q=0$ as $N\to \8$. The free monotone transport theory was extended to the type III setting and was applied to enlarge the isomorphism class of free Araki-Woods factors by the first named author in \cite{Ne15}. Then, we applied the free monotone transport theory to the so-called mixed $q$-Gaussian algebras and enlarged the isomorphism class of the free group factors in \cite{NZ15}. Similar results also hold in the $C^*$-algebra setting. Note that all the later developments were based on the seminal paper \cite{GS14}, and thus work only for a finite number of generators. This leads to the following question:

\begin{question}\label{ques1}
Are the methods of Guionnet and Shlyakhtenko \cite{GS14} valid for an infinite number of variables?
\end{question}

This question has been frequently asked by experts since the paper \cite{GS14} appeared. The difficulty is that the estimates obtained by Dabrowski \cite{Dab} for the upper bounds of $|q|$ (that guarantee the conjugate variables occur as power series in the generators) tend to zero as $N$ tends to infinity. This prevents one from obtaining new isomorphism results for  $L(\fz_\8)$, the von Neumann algebra of the countably generated free group $\fz_\8$.

In this paper, we provide an affirmative answer to Question \ref{ques1} and obtain a new isomorphism result in the context of mixed $q$-Gaussian algebras. Namely, we prove a free monotone transport theorem for infinite variables and use it to show that the mixed $q$-Gaussian von Neumann algebras of infinite generators are all isomorphic to $L(\fz_\8)$, provided the structure array of the mixed $q$-Gaussian algebras has uniformly small entries that decay sufficiently rapidly. Together with the previous works \cites{Dab, GS14}, this provides a reasonably satisfactory answer to Question \ref{ques0} in the regime of small perturbations of free semicircular systems, and gives more evidence towards Voiculescu's vision in \cite{Voi06}.

Let $Z=(Z_n)_{n\in \nz}$ be a sequence of self-adjoint operators in a von Neumann algebra with a faithful normal state. We conduct our analysis on particular Banach $*$-subalgebras  $\ps(Z)^{(R)}$ of $C^*(Z_n\colon n\in \nz)$. Namely, $\ps(Z)^{(R)}$ is the completion of $\cz\lge Z_n\colon n\in \nz\rge$ with respect to a certain norm $\|\cdot\|_R$ indexed by $R>\sup_n \|Z_n\|$, and can be thought of as the set of convergent power series in the $Z_n$, $n\in \nz$, with radii of convergence at least $R$. See Subsection \ref{temp_formalism}.

Our main result (see Theorem \ref{transport_thm}) on free monotone transport can be stated as follows:

\begin{theorem}
Let $X=(X_{n})_{n\in\nz}$ be a free semicircular system. Let $Z=(Z_n)_{n\in\nz}$ be a sequence of self-adjoint operators generating a von Neumann algebra $M$ with a faithful normal state $\psi$. Let $R>\max\{5, \sup_n\|Z_n\|\}$. There is a constant $C(R)>0$ such that if
	\[
		\psi( (Z+\ds W)\cdot P) = \psi\otimes\psi^{op}\circ\Tr( \js P)\qquad \forall P\in \bigoplus_{n\in \nz} \cz\lge Z_n\colon n\in \nz\rge
	\]
for some power series $W=W^* \in \mathscr{P}(Z)^{(R)}$ with $\|W\|_R<C(R)$, then
\begin{align*}
C^*(Z_n: n\in\nz) \cong C^*(X_n: n\in\nz)\quad \text{and}\quad M\cong L(\fz_\8).
\end{align*}
\end{theorem}

In the above theorem, $\ds$ is a noncommutative analogue of the gradient (see Subsection \ref{Free difference quotients, conjugate variables, and cyclic derivatives}), and $\js$ is a noncommutative analogue of the Jacobian (see Subsection \ref{mat_alg}). The formula is a version of the free Gibbs state or Schwinger--Dyson equation introduced in \cite{GS14}.

Let us recall some facts about mixed $q$-Gaussian algebras before stating our isomorphism result. Speicher \cite{Sp93} considered the following commutation relation
\begin{align}\label{qcomm}
l_i^* l_j -q_{ij} l_j l_i^* = \de_{i=j}, \quad i,j\in\nz
\end{align}
for $-1\le q_{ij}=q_{ji}\le 1$. Then Bo\.zejko and Speicher \cite{BS94} showed that \eqref{qcomm} is satisfied by $l_n=l(e_n)$, the creation operators on a Fock space $\fx_Q(\mathcal{H})$ corresponding to an orthonormal basis $(e_n)_{n\in\nz}$ of a real separable Hilbert space $\mathcal{H}$. They also studied other operator algebraic properties in the context of braid relations (a.k.a.\! the Yang--Baxter equation) which in particular include \eqref{qcomm} as a special case. Henceforth we consider $(l_n)$ as operators on $\fx_Q(\mathcal{H})$. Let $Q=(q_{ij})_{i,j\in\nz}$. We call $X_n^Q=l_n+l_n^*, n\in\nz$ the mixed $q$-Gaussian variables and call the von Neumann algebra generated by $(X_n^Q)_{n\in\nz}$ the mixed $q$-Gaussian (von Neumann) algebra, which is denoted by $\Ga_Q$ or $\Ga_Q(\mathcal{H})$ if we want to specify  $\mathcal{H}$. We call $Q =(q_{ij})_{i,j\in\nz}$ the \emph{structure array} of the mixed $q$-Gaussian algebra $\Ga_Q$. The $q$-Gaussian algebra associated to $\mathcal{H}$ for fixed $q_{ij}\equiv q$ is denoted by $\Ga_q(\mathcal{H})$. Let $\ell^p=\ell^p(\nz)$ denote the classical real sequential Banach spaces with the norm
\[
\|(x_n)_{n\in\nz}\|_p = \left(\sum_{n=1}^\8 |x_n|^p \right)^{1/p},\quad 1\le p<\8,
\]
and $\|(x_n)_{n\in\nz}\|_\8 = \sup_{n \in\nz}|x_n|$. We define
	\[
		q_\infty=\sup_{i,j\in\nz} |q_{ij}|.
	\]
Note that this is equal to $\|Q\colon \ell^1\to\ell^{\infty}\|$. For $0<p<\8$, we also define
\begin{equation}\label{qip}
Q_n(p)= \sum_{j\geq 1} |q_{nj}|^p \in [0,+\infty].
\end{equation}

\begin{theorem}\label{thmiso}
Let $R>5$ and define
\[
\pi(Q,n,R)= \frac{[(R(1-q_\8)+1) Q_n(\frac12)]^2}{(1-2q_\8)^2 -[(R(1-q_\8)+1) Q_n(\frac12)]^2}.
\]
Suppose $0<\pi(Q,n,R)<1$ for all $n\in\nz$ and
\[
		\sum_{n=1}^\8 \frac{\pi(Q,n,R)}{1-\pi(Q,n,R)} < \frac{e\log\left(\frac{R}{5}\right)}{R\left( R+ \frac{4}{R-\sup_n \|X_n^Q\|}\right)}.
	\]
Then $C^*(X_n^Q:{n\in\nz})\cong C^*(X_n: n\in \nz)$ and $\Ga_Q(\ell^2)\cong \Ga_0(\ell^2)\cong  L(\fz_\8)$. Here $(X_n: n\in\nz)$ is a free semicircular system.
\end{theorem}

Morally speaking, the conditions for $Q$ in the above theorem require $Q$ to be ``small''. For instance, $q_{ij}=q^{i+j-1}$ with $|q|<0.0002488$ will satisfy the condition for $R=6.7$. On the other hand, some sort of smallness condition seems to be inevitable for isomorphism results. Note that the structure array $(q_{ij}\equiv q)_{i,j\in\nz}$ of $\Ga_q(\ell^2)$ is not even bounded as an operator on $\ell^2$ unless $q=0$. Comparing the isomorphism results in \cite{GS14} and Theorem \ref{thmiso}, it seems reasonable to expect that the smallness condition should take into account the number of generators in the framework of free monotone transport. In other words, assuming only $\sup_{i,j\in\nz} |q_{ij}|$ small may not be sufficient to obtain isomorphism results via free monotone transport.

Let us elaborate on the relationship between our work and previous ones \cites{GS14,Dab,NZ15}. The dimension of the underlying Hilbert space $\mathcal{H}$ for the Bo\.zejko--Speicher algebras plays a non-trivial role in various problems. For instance, it seems easier to prove the factoriality of the Bo\.zejko--Speicher von Neumann algebras when the dimension of $\mathcal{H}$ is infinite; see \cite{BKS, Kr00, Ri05}. However, one immediately faces two major difficulties when considering free monotone transport for infinite generators. First, some quantities used in the $N$-generator case are no longer valid in the infinite generator case. We have to employ some new ideas and techniques to circumvent these quantities and the related argument. For a simple example, in the finite generator case, it is natural to consider the quadratic potential $V(X_1,...,X_N)=\frac12 \sum_{i=1}^N X_i^2$ which characterizes the free semicircular law as a free Gibbs state. This polynomial is a convergent power series with infinite radius of convergence. However, the series $\frac12\sum_{i=1}^\infty X_i^2$ has radius of convergence at most zero. We overcome this difficulty by considering directly the difference of two potentials which is required to be a convergent power series so that it lives in the Banach algebra $\ps(X)^{(R)}$.

Second and more conceptually, when applying the free monotone transport theory to the $q$-Gaussian algebras, one has to construct the conjugate variables to the free difference quotients as considered in \cite{Dab}. This puts serious restrictions on the value of $q$. For instance, an immediate requirement on $q$ is $q^2N<1$. Therefore, it seems impossible to apply the free monotone transport theory to the case $N=\8$. Due to this fact, it might not be plausible to develop this theory for infinite generators in \cite{GS14}. However, in the context of mixed $q$-Gaussian algebras, we have more flexibility to further deform the algebra so that the conjugate variables exist and are close to the mixed $q$-Gaussian variables. This also explains our motivation to extend the free monotone transport theorem to deal with infinite generators.

The paper is organized as follows. In Section 2, we provide some preliminary facts on mixed $q$-Gaussian algebras and free probability theory. We extend various techniques of \cite{GS14} to the setting of infinite generators in Section 3, and use them to construct free monotone transport in Section 4. The isomorphism result for mixed $q$-Gaussian algebras is proved in Section 5.

\section{Preliminaries}

\subsection{Mixed $q$-Gaussian algebras}\label{mixed_q_prelim}

We denote by $\bx(\mathcal{H})$ the bounded operators on a Hilbert space $\mathcal{H}$. Let $\{e_n\}_{n\in\nz}$ denote the standard basis of $\ell^2$: $e_n$ denotes the unit vector supported on $\{n\}$. We will always use an underline to denote a multi-index $\vec{j}=(j_1,...,j_d)\in \nz^d$.

We refer the readers to \cites{BS94, LP99, JZ15} for the following facts on mixed $q$-Gaussian algebras. Let $Q=\{q_{ij}\}_{i,j\in\nz}\subset [-1,1]$ satisfy $q_{ij}=q_{ji}$ for all pairs $i,j\in\nz$. We define an inner product
\[
\lge e_{i_1}\otimes\cdots\otimes e_{i_m}, e_{j_1}\otimes\cdots\otimes e_{j_n} \rge_Q = \de_{m,n} \sum_{\si\in S_n} a(\si,\vec{j}) \lge e_{i_1}, e_{j_{\si^{-1}(1)}}\rge \cdots \lge e_{i_m}, e_{j_{\si^{-1}(n)}}\rge
\]\\
on
	\[
		\cz\Omega\oplus \bigoplus_{d\geq 1} (\ell^2)^{\otimes d}
	\]
and denote the completion by $\fx_Q$. Here $\Om$ is the vacuum state and $a(\si,\vec{j})$ is a product of $(q_{kl})$ whose exact value is not used in this paper; see \cite{LP99,JZ15,NZ15} for the precise value.  Then if $l(e_n)$ denotes the left creation operator
	\begin{align*}
		l(e_n)\Omega &= e_n\\
		l(e_n) e_{i_1}\otimes\cdots\otimes e_{i_d}&= e_n\otimes e_{i_1}\otimes\cdots\otimes e_{i_d},
	\end{align*}
then its adjoint is given by the left annihilation operator
	\begin{align*}
		l(e_n)^*\Omega&=0\\
		l(e_n)^* e_{i_1}\otimes\cdots\otimes e_{i_d} &= \sum_{k=1}^d \delta_{n=i_k} q_{n i_1}\cdots q_{n i_{k-1}} e_{i_1}\otimes\cdots \otimes \hat{e}_{i_k}\otimes\cdots\otimes e_{i_d},
	\end{align*}
where $\hat e_{i_k}$ means that $e_{i_k}$ is omitted in the tensor product. If we write $l_n=l(e_n)$ and $l_n^*=l(e_n)^*$, then these $l_n$, $n\in\nz$, satisfy the commutation relation \eqref{qcomm}. We will also need the right creation operator $r(e_n)$ defined by
\begin{align*}
r(e_n)(e_{i_1}\otimes \cdots \otimes e_{i_d}) = e_{i_1}\otimes \cdots \otimes e_{i_d} \otimes e_n.
\end{align*}
It follows from \cite{BS94} that
	\begin{align}\label{lnorm}
		\|l(e_n)\| = \|r(e_n)\| = \left\{\begin{array}{cl} \frac{1}{\sqrt{1-q_{nn}}} & \text{if }q_{nn}\in [0,1),\\
									1 & \text{if }q_{nn}\in (-1,0]. \end{array}\right.
	\end{align}
For each $n\in\nz$, we denote $X_n^Q=l(e_n)+l(e_n)^*$. Then $X_n^Q$, $n\in\nz$ are the mixed $q$-Gaussian variables and the mixed $q$-Gaussian algebra $\Ga_Q$ (=$\Ga_Q(\ell^2)$) is the von Neumann algebra generated by $\{X_n^Q\colon n\in\nz\}$.

We will always assume $q_\8:=\sup_{i,j\in\nz} |q_{ij}|<1$ in this paper. By \cite{BS94}, there is a normal faithful tracial state on $\Ga_Q(\ell^2)$ given by $\tau_Q(T)=\lge T\Om, \Om\rge_Q$. Moreover, there is a canonical unitary isomorphism between $L^2(\Ga_Q, \tau_Q)$ and $\fx_Q$ which extends continuously from
\[
T\mapsto T\Om,\quad T\in \Ga_Q.
\]
We will write $\lge \cdot, \cdot\rge_{\tau_Q}$ for the inner product of $L^2(\Ga_Q,\tau_Q)$. The Wick word (or Wick product) of $e_{i_1}\otimes\cdots \otimes e_{i_n}$ is denoted by $W(e_{i_1}\otimes\cdots \otimes e_{i_n})$, which is the unique element in $\Ga_Q$ satisfying
\[
W(e_{i_1}\otimes\cdots \otimes e_{i_n}) \Om= e_{i_1}\otimes\cdots \otimes e_{i_n}.
\]
The Wick words can be defined inductively and have a normal form in terms of creation and annihilation operators; see \cites{BKS,Kr00}. According to \cite{BS94}, there exists a strictly positive operator $P^{(n)}$ for each $n\in \nz$ such that
\begin{align}\label{pninn}
\lge \xi,\eta\rge_Q = \de_{n=m} \lge \xi, P^{(n)} \eta\rge_0, \quad \xi\in (\ell^2)^{\otimes n}, \eta\in (\ell^2)^{\otimes m}.
\end{align}
Here $\lge\cdot, \cdot\rge_0$ is the inner product of the full Fock space associated to the free semicircular variables; see \cite{VDN}. By \cite{Boz}*{Theorem 2}, we have
\[
 \|(P^{(n)})^{-1}\|\le \Big[(1-q_\8)\prod_{k=1}^\8\frac{1+q_\8^k}{1-q_\8^k}\Big]^{n}.
\]
When $q_\8<\frac12$, we deduce from the Gauss identity that
\begin{equation}\label{pnorm}
  \|(P^{(n)})^{-1}\|\le \Big[(1-q_\8) \Big(\sum_{k=-\8}^\8(-1)^k q_\8^{k^2}\Big)^{-1}\Big]^{n}\le \Big(\frac{1-q_\8}{1-2q_\8}\Big)^{n}.
\end{equation}


\subsection{Bimodule notations and conventions}\label{bimodule_notations}

Let $\ax$ be a unital algebra, and $\ax^{op}$ its opposite algebra with multiplication $a^{op}\cdot b^{op}=(ba)^{op}$, $a,b\in \ax$. We will often consider algebraic tensor products of the form $\ax\otimes \ax^{op}$, which we view as an $\ax$-$\ax$ bimodule with the action defined by
	\[
		x\cdot (a\otimes b^{op})\cdot y = (xa)\otimes (by)^{op},\qquad x,y,a,b\in \ax
	\]
As an algebra, $\ax\otimes \ax^{op}$ has a natural multiplication defined by
	\[
		( x\otimes y^{op}, a\otimes b^{op}) \mapsto (xa) \otimes (by)^{op}.
	\]
Following \cite{GS14}, we will denote this multiplication with the symbol ``$\#$'' so that we may henceforth repress the ``$op$'' notation from elements in $\ax\otimes \ax^{op}$.

If $\ax$ is a $*$-algebra, with involution $a\mapsto a^*$, then $\ax\otimes \ax^{op}$ inherits three involutions which we denote
	\begin{align*}
		(a\otimes b)^* &= a^* \otimes b^*\\
		(a\otimes b)^\dagger &= b^*\otimes a^*\\
		(a\otimes b)^\diamond & = b\otimes a, \qquad a\otimes b\in \ax\otimes\ax^{op}.
	\end{align*}
We will also use the notation $\diamond(\eta):= \eta^\diamond$ for $\eta\in \ax\otimes \ax^{op}$. If $\phi\colon \ax\to \cz$ is a faithful state, we let $\|x\|_\phi$ denote the GNS (or $L^2$) norm associated to $\phi$ and $\|\cdot\|_{\phi\otimes \phi^{op}}$ denote the norm induced by the inner produced defined by
	\[
		\lge a\otimes b, c\otimes d\rge_{\phi\otimes\phi^{op}} = \phi\otimes \phi^{op}( c^*\otimes d^* \# a\otimes b)=\phi(c^*a)\phi(bd^*).
	\]


\subsection{Free difference quotients, conjugate variables, and cyclic derivatives}\label{Free difference quotients, conjugate variables, and cyclic derivatives}

Suppose $X=(X_n)_{n\in \nz}$ is a sequence of algebraically free self-adjoint operators generating a tracial von Neumann algebra $(M,\tau)$. Let $\ps$ denote the noncommutative polynomials in the $X_n$, $n\in\nz$. After \cite{Voi98}, we define for each $n\in\nz$ the \emph{$n$-th free difference quotient} $\partial_n\colon \ps\to \ps\otimes\ps^{op}$ by
	\begin{align*}
		\partial_n(X_k)&=\delta_{n=k} 1\otimes 1\\
		\partial_n(AB) &= \partial_n(A)\cdot B + A\cdot \partial_n(B),\qquad A,B\in \mathscr{P}.
	\end{align*}
Note that $\partial_n(P)^\dagger = \partial_n(P^*)$ for $P\in \ps$.

The \emph{$n$-th conjugate variable} is a vector $\xi_n\in L^2M$ such that
	\begin{align*}
		\lge P, \xi_n\rge_\tau = \lge \partial_n P, 1\otimes 1\rge_{\tau\otimes\tau^{op}}, \qquad \forall P\in \mathscr{P}.
	\end{align*}
	
\begin{defn}
We say $X=(X_n)_{n\in\nz}$ \emph{has conjugate variables} if the $n$-th conjugate variable exists for every $n\in \nz$.
\end{defn}	

By applying \cite[Theorem 2.5]{MSW17} to finite subsets of $\{X_n\colon n\in\nz\}$, one can drop the condition of algebraic freeness when $X$ has conjugate variables.

Viewing the $n$-th free difference quotient as a densely defined map $\partial_n\colon L^2M\to L^2M\otimes L^2M^{op}$, it is clear that $\xi_n=\partial_n^*(1\otimes 1)$, provided either exists. This characterization leads to the following well-known formula (\emph{cf.} \cite[Proposition 4.1]{Voi98}, the proof of \cite[Theorem 34]{Dab}, or \cite[Corollary 2.4]{Ne15} for the computation.)

\begin{prop}\label{clos_deriv}
Let $\partial$ be a real (i.e. $\partial(P^*)=\partial(P)^\dagger$) derivation with domain $\ax$, a unital $*$-subalgebra of a tracial von Neumann algebra $(M,\tau)$, valued in $\ax\otimes\ax^{op}$. We think of $\partial$ as a densely defined map $\partial\colon L^2 \ax\to L^2M \otimes L^2M^{op}$ with adjoint $\partial^*\colon L^2M\otimes L^2M^{op}\to L^2M$. If $1\otimes 1\in \Dom(\partial^*)$, then $\ax\otimes\ax^{op}\subset \Dom(\partial^*)$ with
	\begin{equation}\label{adjoint_formula}
		\partial^*(\eta) = \eta\#\partial^*(1\otimes 1) - m\circ(1\otimes \tau\otimes 1)\circ(1\otimes\partial + \partial\otimes 1)(\eta),\qquad \eta\in \ax\otimes\ax^{op}
	\end{equation}
where $(a\otimes b)\#c=acb$ and $m(a\otimes b)=ab$. Furthermore, if $\ax$ is dense in $L^2M$ then $\partial$ is closable.
\end{prop}

After \cite{Voi02}, we define for each $n\in\nz$ the \emph{$n$-th cyclic derivative} $\ds_n\colon \ps\to \ps$  by
	\begin{align*}
		\ds_n(p)= \sum_{p=AX_n B} BA,
	\end{align*}
for $p\in\ps$ a monomial (where a sum over the empty set is considered to be zero), and extend linearly to $\ps$. Note that $\ds_n (P)^*=\ds_n(P^*)$ and $\ds_n = m\circ\diamond\circ \partial_n$. For $P\in \ps$, the sequence $\ds P:=(\ds_n P)_{n\in \nz}$ is called the \emph{cyclic gradient of $P$}.


\subsection{Free transport}

Let $X=(X_n)_{n\in\nz}$ be a sequence of self-adjoint operators in a von Neumann algebra $M$ with a faithful normal state $\varphi$, which we regard as a noncommutative probability space $(M,\varphi)$. The \emph{joint law of $X$ with respect to $\varphi$} is a linear functional $\varphi_X$ defined on polynomials in non-commuting self-adjoint indeterminates $\{T_n\colon n\in\nz\}$ by
	\[
		\varphi_X(T_{i_1}\cdots T_{i_d}) = \varphi(X_{i_1}\cdots X_{i_d}),\qquad \forall \vec{i}\in \nz^d.
	\]

Let $Z=(Z_n)_{n\in\nz}$ be another sequence in a von Neumann algebra $N$ with a faithful normal state $\psi$, and let $\psi_Z$ be the joint law of $Z$ with respect to $\psi$. Observe that if $\varphi_X=\psi_Z$ then
	\[
		W^*(X_n\colon n\in\nz) \cong W^*(Z_n\colon n\in\nz),
	\]
since $\varphi$ and $\psi$ are faithful normal states. After \cite{GS14}, we make the following definition:

\begin{defn}
\emph{Transport from $\varphi_X$ to $\psi_Z$} is a sequence $Y=(Y_n)_{n\in\nz}\subset W^*(X_n\colon n\in\nz)$ whose joint law with respect to $\varphi$ is equal to $\psi_Z$. That is, $\varphi_Y=\psi_Z$. We call such transport \emph{monotone} if $Y-X$ belongs to the $\ell^2(\nz, L^2(M,\varphi))$-closure of the set
	\[
		\{ \ds g\colon g\in \ps \text{ and } 1+(\partial_j \ds_i g)_{i,j\in\nz} \geq 0\},
	\]
where $(\partial_j \ds_i g)_{i,j\in\nz}\in (M\bar{\otimes} M^{op})^{\nz^2}$ is regarded as an operator on $\ell^2(\nz, L^2(M\bar{\otimes}M^{op},\varphi\otimes\varphi^{op}))$ defined by
	\[
		\ell^2(\nz, L^2(M\bar{\otimes}M^{op},\varphi\otimes\varphi^{op}))\ni (\xi_n)_{n\in \nz}\mapsto \left( \sum_j (\partial_j \ds_i g)\xi_j\right)_{i\in \nz}
	\]
(\emph{cf.} Proposition \ref{prop:M_1_bdd} for why such an action is bounded, but also note that all but finitely many of the entries $\partial_j \ds_i G$ will be zero since $G$ is a polynomial).
\end{defn}

By the above observation, whenever $Y$ is transport from $\varphi_X$ to $\psi_Z$ we have an embedding of $W^*(Z_n\colon n\in\nz)$ into $W^*(X_n\colon n\in\nz)$ by extending the map $Z_{i_1}\cdots Z_{i_d}\mapsto Y_{i_1}\cdots Y_{i_d}$, $\vec{i}\in\nz^d$.

In the classical case, transport refers to a measurable map $T$ from one probability space $(\mathcal{X},\mu)$ to another $(\mathcal{Z},\nu)$ such that $T_*\mu=\nu$. Consequently, $f\mapsto f\circ T$ is an embedding of $L^\infty (\mathcal{Z},\nu)$ into $L^\infty(\mathcal{X},\mu)$. The definition in the free case is in analogy with this observation.

In the finite variable case, the monotonicity condition is equivalent to $Y$ (rather than $Y-X$) being in the $L^2(M,\varphi)^N$-closure of the set
	\[
		\{ \ds G\colon G\in \ps \text{ and } (\partial_j \ds_i G)_{i,j=1}^N \geq 0\}.
	\]
The modification in the infinite variable case is a consequence of the identity operator $1$ not being realized as $(\partial_j\ds_i G)_{i,j}$ for any $G\in \ps$. Rather, it occurs when $G$ is the formal power series $V_0=\frac12 \sum_{i=1}^\infty X_i^2$ (\emph{cf.} Remark \ref{V_0_def}).


\subsection{Some (temporary) formalism}\label{temp_formalism}

We consider $T=(T_n)_{n\in\nz}$, a sequence of non-commuting self-adjoint indeterminates, and denote by $\ps(T)$ the noncommutative polynomials in these indeterminates. For a polynomial $P$ and a monomial $p$ in $\ps(T)$, let $c_p(P)$ denote the coefficient of the monomial $p$ in $P$. After \cite{GS14}, for each $R>0$ we define for $P\in \ps(T)$
	\begin{align*}
		\left\| P\right\|_R = \sum_{p} |c_p(P)| R^{\deg(p)},
	\end{align*}
where the sum over monomials $p$ is finite for polynomials. Due to the algebraic freeness of the indeterminates, each $\|\cdot\|_R$ is a norm. We let $\ps(T)^{(R)}$ denote the completion of $\ps(T)$ with respect to this norm, which is easily seen to be a Banach algebra and may be identified as a subalgebra of the formal power series in the $T_n$, $n\in\nz$.

\begin{rem}\label{polynomials_supported_finitely}
Note that elements in $\ps(T)^{(R)}$ are approximated by polynomials, which are each necessarily elements of an algebra generated by finitely many indeterminates $T_i$. As a consequence, many of the arguments from \cite{Dab} and \cite{GS14} adapt easily to our infinitely generated context.
\end{rem}

After \cite{GS14} we also consider the following maps on $\ps(T)^{(R)}$, for $R>0$. The \emph{number operator} $\ns$  is defined on monomials by
	\begin{align*}
		\ns T_{i_1}\cdots T_{i_d} = d T_{i_1}\cdots T_{i_d},
	\end{align*}
and extended linearly to $\ps(T)^{(R)}$. We let $\Pi\colon P\mapsto P-P(0)$, the projection onto power series with zero constant term. We let $\Sigma$ be the inverse of $\ns$ on $\Pi \ps(T)^{(R)}$:
	\begin{align*}
		\Sigma T_{i_1}\cdots T_{i_d} = \frac{1}{d} T_{i_1}\cdots T_{i_d},
	\end{align*}
which we extend to $\ps(T)^{(R)}$ by letting it annihilate constants. Finally, we let $\ss$ denote the map on $\Pi \ps(T)^{(R)}$ defined on monomials by
	\begin{align*}
		\ss T_{i_1}\cdots T_{i_d} = \frac{1}{d} \sum_{k=1}^d T_{i_{k+1}}\cdots T_{i_d} T_{i_1}\cdots T_{i_k}.
	\end{align*}
After \cite{GS14}, we say the elements in the range of $\ss$ are \emph{cyclically symmetric}; that is, for every $\vec{i}\in \nz^d$ the coefficient of the monomial $T_{i_1}\cdots T_{i_d}$ is the same as the coefficient of $T_{i_d}T_{i_1}\cdots T_{i_{d-1}}$. Observe that $\ss$ is contractive with respect to the $\|\cdot\|_R$-norm.

We note that if $Y=(Y_n)_{n\in \nz}$ is a sequence of elements in a Banach algebra $\bx$ and $R\geq \sup_i \|Y_i\|_{\bx}$, then $T_{i_1}\cdots T_{i_n}\mapsto Y_{i_1}\cdots Y_{i_n}$ extends to a contractive map on $\ps(T)^{(R)}$. For $P\in \ps(T)^{(R)}$, we shall let $P(Y)$ denote the image of $P$ under this map, and let $\ps(Y)^{(R)}$ denote the range of the map. In general, this map may fail to be an embedding, but when it is $\|P(Y)\|_R:=\|P\|_R$ defines a norm on $\ps(Y)^{(R)}\subset \bx$.

It follows from \cite[Lemma 37]{Dab} that if $\bx$ is a tracial von Neumann algebra and we consider finite sequences $T=(T_1,\ldots, T_N)$ and $Y=(Y_1,\ldots, Y_N)$,  then the aforementioned contractive mapping for $R>\max_{1\leq i\leq N} \|Y_i\|$ is an embedding provided the conjugate variables to $Y_1,\ldots, Y_N$ exist. It is easy to see that this proof adapts to case of infinite sequences $T=(T_n)_{n\in\nz}$ and $Y=(Y_n)_{n\in\nz}$, especially in light of Remark \ref{polynomials_supported_finitely}, and so we record the following corollary:

\begin{cor}\label{R_embedding}
Suppose $X=(X_n)_{n\in \nz}$ is a uniformly bounded sequence of self-adjoint operators in a tracial noncommutative probability space $(M,\tau)$. If the $n$-th conjugate variable $\xi_n$ exists for all $n\in \nz$ and $R> \sup_n \|X_n\|$, then the map $\ps(T)^{(R)}\ni P\mapsto P(X)\in M$ is an embedding and consequently $\|P(X)\|_R:=\|P\|_R$ is a norm on $\ps(X)^{(R)}$.
\end{cor}

We also note that in the context of the above corollary, one always has $\|P(X)\|_\tau\leq \|P(X)\|\leq\|P(X)\|_R$ for all $P\in \ps(T)^{(R)}$. The hypotheses of this corollary will always be satisfied in Sections 3 and 4, and so we abandon the formal indeterminates $T$ until Section 5.

\section{The infinite variable formalism}\label{Notation}

We now fix $X=(X_n)_{n\in \nz}$ a uniformly bounded sequence of self-adjoint operators that generates a tracial von Neumann algebra $(M,\tau)$ and has conjugate variables. We write $\ps=\ps(X)$ and $\ps^{(R)}=\ps(X)^{(R)}$ for $R>\sup_n \|X_n\|$, and $R$ will always be assumed to be a number satisfying this inequality.


\subsection{Various norms}

Let $\ps^{(R)}_\infty$ denote $\ell^\infty(\nz,\ps^{(R)})$, the set of uniformly bounded sequences of elements of $\ps^{(R)}$, with norm
	\[
		\|(P_n)_{n\in\nz}\|_{R,\infty}:=\sup_n \|P_n\|_R.
	\]
We let $\ps^{(R)}_1$ denote $\ell^1(\nz,\ps^{(R)})$, the set of absolutely summable sequences of elements in $\ps^{(R)}$, with norm
	\[
		\|(P_n)_{n\in \nz}\|_{R,1}:= \sum_{n\in\nz} \|P_n\|_R.
	\]
Observe $\ps^{(R)}_1\subset\ps^{(R)}_\infty$, and that for $F\in \ps^{(R)}_\infty$ and $G\in \ps^{(R)}_1$
	\begin{align*}
		F\# G :=\sum_{n\in\nz} F_nG_n\qquad\text{and}\qquad G\# F := \sum_{n\in \nz} G_n F_n
	\end{align*}
define elements in $\ps^{(R)}$.

On $\ps^{(R)}\otimes (\ps^{(R)})^{op}$, we let $\|\cdot\|_{R\otimes_\pi R}$ denote the projective tensor norm:
	\begin{align*}
		\| \eta \|_{R\otimes_\pi R}:= \inf\left\{ \sum_i \|A_i\|_R\|B_i\|_R\colon \eta = \sum_i A_i\otimes B_i\right\}.
	\end{align*}
Let $\ps \hat{\otimes}_R \ps^{op}$ denote the completion of $\ps^{(R)}\otimes(\ps^{(R)})^{op}$ with respect to this norm. Observe that $\ps\otimes\ps^{op}$ is dense in $\ps\hat\otimes_R\ps^{op}$. Extending the multiplication $\#$ from Subsection \ref{bimodule_notations}, $\ps\hat\otimes_R\ps^{op}$ becomes a Banach algebra. It is also clear that the operation $(a\otimes b)\# c=acb$ extends to $\ps\hat\otimes_R \ps^{op}$ and satisfies
	\[
		\|\eta\# P\|_R\leq \|\eta\|_{R\otimes_\pi R} \|P\|_R
	\]
for $\eta\in \ps\hat\otimes_R\ps^{op}$ and $P\in \ps^{(R)}$. Additionally, since $P\mapsto P^*$ extends to an isometry on $\ps^{(R)}$, it follows that the three involutions $*$,$\dagger$, and $\diamond$ from Subsection \ref{bimodule_notations} extend to isometries on $\ps\hat{\otimes}_R\ps^{op}$, which we denote in the same way.

We also have $\ps\hat\otimes_R\ps^{op}\subset M\bar{\otimes} M^{op}$ (the von Neumann tensor product). Indeed, since $\|\cdot\|_R$ dominates the operator norm, we have $\ps^{(R)}\otimes (\ps^{(R)})^{op}\subset M\otimes M^{op}$. For $\eta\in \ps^{(R)}\otimes (\ps^{(R)})^{op}$ and $\epsilon>0$ let $\sum_i A_i\otimes B_i=\eta$ such that
	\begin{align*}
		\sum_i \|A_i\|_R\|B_i\|_R \leq \|\eta\|_{R\otimes_\pi R}+\epsilon.
	\end{align*}
Then
	\begin{align*}
		\|\eta\|_{M\bar{\otimes}M^{op}} \leq \sum_i \|A_i\| \|B_i\| \leq \sum_i \|A_i\|_R \|B_i\|_R \leq \|\eta\|_{R\otimes_\pi R}+\epsilon.
	\end{align*}
Since $\epsilon$ was arbitrary, we see that $\|\cdot\|_{M\bar\otimes M^{op}}\leq \|\cdot \|_{R\otimes_\pi R}$ and obtain the desired inclusion.  Note also that on $\ps\hat\otimes_R\ps^{op}$ we have $\|\cdot\|_{\tau\otimes \tau^{op}} \le \|\cdot\|_{M\bar\otimes M^{op}} \le \|\cdot\|_{R\otimes_\pi R}$.


\subsection{Bounds on the differential operators}

Recall that by Proposition \ref{clos_deriv} if the conjugate variable $\xi_n$ exists, then $\partial^*_n$ is densely defined and $\partial_n$ is closable as a densely defined map $\partial_n\colon L^2M \to L^2M\otimes L^2M^{op}$. Let $\bar\partial_n$ denote its closure. Given $R>S>0$, we will denote throughout the paper
	\[
		\mathscr{C}(R,S):= \sup_{t>0} t \frac{S^{t-1}}{R^t}=\frac{1}{ e S\log(R/S)}.
	\]

\begin{lemma}\label{free_difference_quotients_bounded}
Let $R>S>\sup_n \|X_n\|$. Then for every subset $I\subset \nz$, and $P\in \ps$
	\begin{align*}
		\sum_{n\in I} \|\partial_n\Sigma P\|_{R\otimes_\pi R} \leq \frac{1}{R} \|P\|_R \quad \text{and}\quad		\sum_{n\in I} \|\partial_n P\|_{S\otimes_\pi S}\leq \mathscr{C}(R,S) \|P\|_R,
	\end{align*}
Hence, $\ps^{(R)}\subset \Dom(\bar{\partial}_n)$ for each $n$ and
	\begin{align*}
		\left\| \sum_{n\in I} \bar\partial_n\Sigma \colon \ps^{(R)}\to \ps\hat\otimes_R\ps^{op}\right\| \le \frac1R \quad \text{and} \quad
		\left\| \sum_{n\in I} \bar\partial_n \colon \ps^{(R)}\to \ps\hat\otimes_S\ps^{op}\right\| \le \mathscr{C}(R,S).
	\end{align*}
\end{lemma}
\begin{proof}
Of the first two estimates we prove only the second, as the first will be obvious from the proof. Given $P\in \ps$ of degree $d$, write $P=\sum_{k=1}^d \sum_{\vec{i}\in \nz^k} c_{\vec{i}} X_{i_1}\cdots X_{i_k}$. Let $J\subset \nz$ be the finite set of indices $j$ such that $c_{\vec{i}}=0$ unless $\vec{i}\in J^k$ for some $k$. Thus we may write
	\[
		P=\sum_{k=1}^d \sum_{\vec{j}\in J^k} c_{\vec{j}} X_{j_1}\cdots X_{j_k}.
	\]
Then for any subset $I\subset \nz$ we have
	\begin{align*}
		\sum_{n\in I} \|\partial_n P\|_{S\otimes_\pi S} &\leq \sum_{k=1}^d \sum_{\vec{j}\in J^k}|c_{\vec{j}}| \sum_{n\in I} \|\partial_n X_{j_1}\cdots X_{j_k}\|_{S\otimes_\pi S}\\
			&\leq \sum_{k=1}^d \sum_{\vec{j}\in J^k} |c_{\vec{j}}| \sum_{l=1}^k \chi_I(j_l) S^{k-1}\\
			&\leq \sum_{k=1}^d \sum_{\vec{j}\in J^k} |c_{\vec{j}}| k \frac{S^{k-1}}{R^k} R^k\\
			&\leq \mathscr{C}(R,S) \sum_{k=1}^d \sum_{\vec{j}\in J^k} |c_{\vec{j}}| R^k\\
			& = \mathscr{C}(R,S)\|P\|_R.
	\end{align*}
This yields the desired estimate. The claims regarding $\bar\partial_n$ then follow easily from the observations that $\|\cdot\|_\tau \leq \|\cdot\|_R$ on $\ps^{(R)}$ and $\|\cdot\|_{\tau\otimes\tau^{op}} \leq \|\cdot\|_{R\otimes_\pi R}$ on $\ps\hat\otimes_R\ps^{op}$.
\end{proof}

\begin{lemma}\label{1_tensor_tau}
Let $R>\sup_n \|X_n\|$, then for every $I\subset \nz$ and $P\in \ps$ we have
	\begin{align*}
		\sum_{n\in I}& \| (1\otimes \tau)(\partial_n P)\|_R \leq \frac{1}{R-\sup_n \|X_n\|} \|P\|_R,\\
		\sum_{n\in I}& \| (\tau \otimes 1)(\partial_n P)\|_R \leq \frac{1}{R-\sup_n \|X_n\|} \|P\|_R.
	\end{align*}
\end{lemma}
\begin{proof}
Let $S:=\sup_n\|X_n\|$. Given $P\in \ps$ of degree $d$, let $J\subset \nz$ be as in the proof of Lemma \ref{free_difference_quotients_bounded}. Then
	\begin{align*}
		\sum_{n\in I} \|(1\otimes \tau)(\partial_n P)\|_R &\leq \sum_{k=1}^d \sum_{\vec{j}\in J^k}|c_{X_{j_1}\cdots X_{j_k}}(P)| \sum_{n\in I} \|(1\otimes \tau)(\partial_n X_{j_1}\cdots X_{j_k})\|_{R}\\
			& \leq \sum_{k=1}^d \sum_{\vec{j}\in J^k}|c_{X_{j_1}\cdots X_{j_k}}(P)|\sum_{l=1}^k \chi_I(j_l) R^{l-1} |\tau(X_{j_{l+1}}\cdots X_{j_k})|\\
			& \leq \sum_{k=1}^d \sum_{\vec{j}\in J^k} |c_{X_{j_1}\cdots X_{j_k}}(P)| R^k \frac{1}{R}\sum_{l=1}^k \frac{S^{k-l}}{R^{k-l}}\\
			& \leq \|P\|_R \frac{1}{R-S}.
	\end{align*}
The computation for $(\tau\otimes 1)$ is similar.
\end{proof}

\begin{cor}\label{free_difference_quotient_adjoint_bounded}
Let $R>\sup_n \|X_n\|$. Assume for each $n\in \nz$, that the conjugate variable $\xi_n$ is given by an element of $\ps^{(R)}$. Then, thinking of the free difference quotients as densely defined maps
	\[
		\partial_n\colon L^2M\to L^2M\otimes L^2M^{op},
	\]
we have $\ps\hat\otimes_R\ps^{op}\subset \Dom(\partial_n^*)$ and
	\[
		\left\|\partial_n^*\colon \ps\hat\otimes_R\ps^{op}\to \ps^{(R)}\right\| \leq \frac{2}{R-\sup_n\|X_n\|}+\|\xi_n\|_R.
	\]
\end{cor}
\begin{proof}
Let $\eta=\sum_i A_i\otimes B_i\in \ps\otimes\ps^{op}$. Then (\ref{adjoint_formula}) implies
	\[
		\partial_n^*(\eta) = \eta\# \xi_n - \sum_i (1\otimes\tau)(\partial_n A_i) B_i - A_i(\tau\otimes 1)(\partial_n B_i),
	\]
and the right-hand side defines an element of $\ps^{(R)}$ by our assumptions on $\xi_n$. Then Lemma \ref{1_tensor_tau} implies
	\begin{align*}
		\|\partial_n^*(\eta)\|_R &\leq \| \eta\|_{R\otimes_\pi R} \|\xi_n\|_R + \frac{2}{R-\sup_{n\in\nz} \|X_n\|}\sum_i \| A_i\|_R\|B_i\|_R .
	\end{align*}
Since $\sum_i A_i\otimes B_i$ can be chosen so that $\sum_i \|A_i\|_R\|B_i\|_R$ is arbitrarily close $\|\eta\|_{R\otimes_\pi R}$, we have the claimed bound on the norm. Then using the density of $\ps\otimes\ps^{op}$ in $\ps\hat\otimes_R\ps^{op}$ and the inequalities $\|\cdot\|_\tau\leq \|\cdot\|_R$, $\|\cdot\|_{\tau\otimes\tau^{op}}\leq \|\cdot\|_{R\otimes_\pi R}$, we see that $\ps\hat\otimes_R\ps^{op}\subset\Dom(\partial_n^*)$ with the claimed bound. 
\end{proof}

The following lemma follows from the same argument as in Lemma \ref{free_difference_quotients_bounded}.

\begin{lemma}\label{cyclic_derivative_bounded}
Let $R>S>\sup_n\|X_n\|$. Then for any subset $I\subset \nz$, $\sum_{n\in I} \ds_n \Sigma$ and $\sum_{n\in I} \ds_n$ extend to bounded maps
	\begin{align*}
		\left\| \sum_{n\in I} \ds_n \Sigma \colon \ps^{(R)}\to \ps^{(R)}\right\| \leq \frac{1}{R}\qquad\text{and}\qquad \left\|\sum_{n\in I} \ds_n \colon \ps^{(R)}\to \ps^{(S)}\right\|\le \mathscr{C}(R,S).
	\end{align*}
Consequently, we have
	\begin{align*}
		\left\|\ds\Sigma\colon \ps^{(R)}\to \ps^{(R)}_1\right\|\leq \frac{1}{R} \qquad\text{and}\qquad \left\|\ds\colon \ps^{(R)}\to \ps^{(S)}_1\right\|\le \mathscr{C}(R,S).
	\end{align*}
\end{lemma}


\subsection{A mean value estimate and inverse function theorem}

Let $R,S>2$, $P\in \ps^{(R)}$, $Y\in\ps^{(S)}_{\infty}$. If $\|Y\|_{S,\infty}\leq R$, then $P(Y)$, by which we mean $P$ with $X_n$ replaced by $Y_n$ for every $n\in\nz$, is a convergent power series with radius of convergence at least $S$. Moreover, $\|P(Y)\|_S\leq \|P\|_R$. The following lemma (which is a folklore result whose proof we provide for the convenience of the reader) provides us with a ``mean-value'' estimate.

\begin{lemma}\label{mean_value_estimate}
Let $R,S>\sup_n\|X_n\|$, $P\in \ps^{(R+\epsilon)}$ for some $\epsilon>0$, and $Y^{(j)}\in \ps^{(S)}_\infty$ satisfying $\|Y^{(j)}\|_{S,\infty}\leq R$, $j=1,2$. Then
	\[
		\| P(Y^{(1)}) - P(Y^{(2)})\|_S\leq \sum_{n\in \nz} \|\bar\partial_n P\|_{R\otimes_\pi R} \| Y^{(1)} - Y^{(2)}\|_S,
	\]
and consequently
	\[
		\| P(Y^{(1)}) - P(Y^{(2)})\|_S \leq \mathscr{C}(R+\epsilon, R)\|P\|_{R+\epsilon}\|Y^{(1)} - Y^{(2)}\|_S.
	\]
\end{lemma}
\begin{proof}
By Lemma \ref{free_difference_quotients_bounded}, $\sum_{n\in \nz} \bar\partial_n P\in \ps\hat\otimes_R\ps^{op}$. Since $\|Y^{(j)}\|_{S,\infty}\leq R$ for $j=1,2$,
	\[
		\sum_{j\in \nz} [\bar\partial_j P](Y^{(1)}, Y^{(2)})\in \ps\hat\otimes_S\ps^{op},
	\]
where for $A,B\in\ps^{(R)}$ we denote $[A\otimes B](Y^{(1)}, Y^{(2)}) = A(Y^{(1)}) \otimes B(Y^{(2)})$.

Now, suppose $P=\sum_{d\geq 0}\sum_{\vec{i}\in \nz^d} c_{\vec{i}} X_{i_1}\cdots X_{i_d}$.  Then
	\begin{align*}
		\|P(Y^{(1)}) - P(Y^{(2)}) \|_S &= \left\| \sum_{\substack{d\geq 1\\ \vec{i}\in \nz^d}} c_{\vec{i}} \left\{\sum_{k=1}^d  Y^{(1)}_{i_1}\cdots Y^{(1)}_{i_{k-1}}\left[ Y^{(1)}_{i_k} - Y^{(2)}_{i_k}\right] Y^{(2)}_{i_{k+1}} \cdots Y^{(2)}_{i_d}\right\}\right\|_S\\
		&=\left\| \sum_{\substack{d\geq 1\\ \vec{i}\in \nz^d}} c_{\vec{i}} \left\{ \sum_{n\in \nz} [\partial_n X_{i_1}\cdots X_{i_d}](Y^{(1)}, Y^{(2)}) \# \left[Y^{(1)}_n - Y^{(2)}_n\right]\right\} \right\|_S\\
		&=\left\| \sum_{n\in \nz} [\bar\partial_n P](Y^{(1)}, Y^{(2)}) \# \left[Y^{(1)}_n - Y^{(2)}_n\right]\right\|_S\\
		&\leq \sum_{n\in \nz} \| [\bar\partial_n P](Y^{(1)}, Y^{(2)})\|_{S\otimes_\pi S} \left\|Y^{(1)}_n - Y^{(2)}_n\right\|_S\\
		&\leq \sum_{n\in \nz} \|\bar\partial_n P\|_{R\otimes_\pi R} \| Y^{(1)} - Y^{(2)}\|_S.
	\end{align*}
Appealing to Lemma \ref{free_difference_quotients_bounded} then concludes the proof.
\end{proof}

The following inverse function theorem for $\ps^{(R)}_\infty$ is the analogue of \cite[Corollary 2.4]{GS14}, and the proof is similar.

\begin{lemma}\label{inverse_function}
Let $R>S>\sup_n \|X_n\|$. Then there is a constant $0<C<R-S$, depending only on $R$ and $S$, such that if $f\in \ps^{(R)}_\infty$ satisfies $\|f\|_{R,\infty}<C$ then there exists $H\in \ps^{(S+C)}_\infty$ satisfying $H(X+f)=X$.
\end{lemma}
\begin{proof}
Let $T=\frac{R+S}{2}$ and choose $C$ small enough so that
	\[
		C_1:= C\cdot \mathscr{C}(R,T)<1
	\]
and
	\[
		S+C+\frac{C}{1-C_1} \leq T.
	\]
Assume $\|f\|_{R,\infty}<C$.

We recursively define $H^{(k)}=X- f(H^{(k-1)})$ with $H^{(0)}=X$, and note that \emph{a priori} for each $k,n\in\nz$, $H^{(k)}_n$ is a formal power series. However, we claim that for every $k\in \nz$, $\|H^{(k)}\|_{S+C,\infty}\leq T$. Since $S+C<T$, this holds for $H^{(0)}$. Inductively assume that this inequality is satisfied by $H^{(0)},\ldots, H^{(k-1)}$. For each $n\in\nz$ and $0\leq l\leq k-1$, we use Lemma \ref{mean_value_estimate} to see
	\begin{align*}
		\| H^{(l+1)}_n - H^{(l)}_n\|_{S+C} &= \| f_n(H^{(l)}) - f_n(H^{(l-1)})\|_{S+C}\\
			&\leq \mathscr{C}(R, T) \|f\|_{R,\infty} \| H^{(l)} - H^{(l-1)}\|_{S+C}\\
			&< C_1 \| H^{(l)} - H^{(l-1)}\|_{S+C}\\
			&< (C_1)^l \| H^{(1)} - H^{(0)}\|_{S+C} = (C_1)^l \|f\|_{S+C}< C (C_1)^l.
	\end{align*}
Since $n\in \nz$ was arbitrary, we have the same inequality for $\| H^{(l+1)} - H^{(l)}\|_{S+C,\infty}$. Therefore
	\begin{align*}
		\|H^{(k)}\|_{S+C,\infty} &\leq \|H^{(0)}\|_{S+C,\infty}+ \|H^{(k)} - H^{(0)}\|_{S+C,\infty}\\
			&< S+C + C \sum_{l=0}^{k-1} (C_1)^l\\
			&< S+C + \frac{C}{1-C_1}\leq T.
	\end{align*}
So by induction we have $\|H^{(k)}\|_{S+C,\infty}\leq T$ for every $k\in \nz$. In fact, the proof showed that $(H^{(k)})_{k\in \nz}$ is a Cauchy sequence in $\ps^{(S+C)}_\infty$. Let $H\in \ps^{(S+C)}_\infty$ be the limit point, in which case $H=X- f(H)$ and $\|H\|_{S+C,\infty}\leq T$. Hence
	\begin{align*}
		\| H(X+f) - X\|_{S,\infty} &= \|f - f(H(X+f))\|_{S,\infty}\\
						& \leq \mathscr{C}(R,T) \|f\|_{R,\infty} \| X- H(X+f)\|_{S,\infty}.
	\end{align*}
Since $\mathscr{C}(R,T) \|f\|_{R,\infty}<1$, it must be that $\|H(X+f)-X\|_{S,\infty}=0$.
\end{proof}

\begin{rem}\label{inverse_function_constant}
For an explicit formula for $C$ in the previous lemma, we can take
	\[
		C=\frac{\frac{R-S}{4}T e\log\left(\frac{R}{T}\right)}{ T e\log\left(\frac{R}{T}\right) + \frac{R-S}{4}};
	\]
that is, the solution of the equation $\frac{C}{1- C\mathscr{C}(R,T)} = \frac{R-S}{4}$. 
\end{rem}


\subsection{Matrix algebras over $\ps\hat\otimes_R\ps^{op}$}\label{mat_alg}

We consider the Banach spaces
	\begin{align*}
		\mathcal{M}_\infty(\ps\hat\otimes_R\ps^{op})&:=\left\{ H\in (\ps\hat\otimes_R\ps^{op})^{\nz^2}\colon \sup_{i\in \nz}\sum_{j\in \nz} \|H(i,j)\|_{R\otimes_\pi R}<\infty\right\},\ \text{and}\\
		\mathcal{M}_1(\ps\hat\otimes_R\ps^{op})&:=\left\{ H\in (\ps\hat\otimes_R\ps^{op})^{\nz^2}\colon \sum_{i,j\in\nz} \|H(i,j)\|_{R\otimes_\pi R}<\infty\right\},
	\end{align*}
with norms denoted by
	\begin{align*}
		\|H\|_{R,\infty,1} &:= \sup_{i\in \nz}\sum_{j\in \nz} \|H(i,j)\|_{R\otimes_\pi R},\ \text{and}\\
		\|H\|_{R,1,1} &:= \sum_{i,j\in\nz} \|H(i,j)\|_{R\otimes_\pi R},
	\end{align*}
respectively. We observe
	\[
		\mx_1(\ps\hat\otimes_R\ps^{op})\subset \mx_\infty(\ps\hat\otimes_R\ps^{op}).
	\]

On $\mx_\infty(\ps\hat\otimes_R\ps^{op})$, we define a product by
	\[
		[GH](i,j)= \sum_{k\in \nz} G(i,k)\# H(k,j),
	\]
and note that both $\mx_\infty(\ps\hat\otimes_R\ps^{op})$ and $\mx_1(\ps\hat\otimes_R\ps^{op})$ are Banach algebras with this product. Furthermore, $\mx_1(\ps\hat\otimes_R\ps^{op})$ is a right ideal in $\mx_\infty(\ps\hat\otimes_R\ps^{op})$.

We define an involution $\dagger$ on $\mx_\infty(\ps\hat\otimes_R\ps^{op})$ by
	\[
		[H^\dagger](i,j)=H(i,j)^\dagger.
	\]
On $\mx_1(\ps\hat\otimes_R\ps^{op})$ we also have involutions $*$ and $\diamond$ defined by
	\begin{align}
		[H^*](i,j)&=H(j,i)^* \label{hadjoint}\\
		[H^\diamond](i,j)&= H(j,i)^\diamond, \nonumber
	\end{align}
so that, for example, $H^\dagger = (H^*)^\diamond$. These maps can clearly be defined for $\mx_\infty(\ps\hat\otimes_R\ps^{op})$, but the images may no longer lie in $\mx_\infty(\ps\hat\otimes_R\ps^{op})$ since the roles of $i$ and $j$ are switched. However, if $H(i,j)=H(j,i)^*$ (resp. $H(i,j)=H(j,i)^\diamond$) for $H\in \mx_\infty(\ps\hat\otimes_R\ps^{op})$, then $H^*$ (resp. $H^\diamond$) is well-defined in $\mx_\infty(\ps\hat\otimes_R\ps^{op})$.

We also define
	\begin{align*}
		\text{Tr}\colon \mx_1(\ps\hat\otimes_R\ps^{op})\to \ps\hat\otimes_R \ps^{op}
	\end{align*}
by
	\[
		\text{Tr}(H):=\sum_{i\in\nz} H(i,i),
	\]
and note that $\text{Tr}$ is bounded, linear, and satisfies
	\[
		(\tau\otimes \tau^{op})\circ\text{Tr}(GH)=(\tau\otimes\tau^{op})\circ\text{Tr}(HG)
	\]
for all $G,H\in \mx_1(\ps\hat\otimes_R\ps^{op})$.

For $H\in \mx_\infty(\ps\hat\otimes_R\ps^{op})$, $\eta=(\eta_j)_{j\in\nz}\in \ell^\infty(\nz, \ps\hat\otimes_R\ps^{op})$, and $P=(P_j)_{j\in\nz}\in \ps^{(R)}_\infty$ we write
	\begin{align*}
		H\# \eta &:= \left( \sum_{j\in \nz} H(i,j)\# \eta_j\right)_{i\in\nz}\text{, and}\\
		H\# P &:= \left( \sum_{j\in \nz} H(i,j)\# P_j \right)_{i\in \nz}.
	\end{align*}
It is easy to see that this defines bounded actions of $\mx_\infty(\ps\hat\otimes_R\ps^{op})$ on $\ell^\infty(\nz,\ps\hat\otimes_R\ps^{op})$ and $\ps^{(R)}_\infty$, where the action of $H$ is bounded in both cases by $\|H\|_{R,\infty, 1}$. 

We define the same action for $H\in \mx_1(\ps\hat\otimes_R\ps^{op})$, though, as the following propositions indicate, we can reach much stronger conclusions.

\begin{prop}
For $H\in \mx_1(\ps\hat\otimes_R\ps^{op})$ the above action defines bounded maps:
	\begin{align*}
		\left\| H\# \colon \ell^\infty(\nz, \ps\hat\otimes_R\ps^{op})\to \ell^1(\nz, \ps\hat\otimes_R\ps^{op})\right\| &\leq \|H\|_{R,1,1},\qquad\text{and}\\
			\left\| H\#\colon \ps^{(R)}_\infty\to \ps^{(R)}_1\right\|&\leq \|H\|_{R,1,1}.
	\end{align*}
In particular, $H$ acts boundedly on $\ell^1(\nz, \ps\hat\otimes_R\ps^{op})$ and $\ps^{(R)}_1$.
\end{prop}
\begin{proof}
Let $P=(P_j)_{j\in \nz}\in \ps_\infty^{(R)}$. Then
	\begin{align*}
		\| H\# P\|_{R,1} &= \sum_{i\in \nz} \left\| \sum_{j\in \nz} H(i,j)\# P_j \right\|_R \leq \sum_{i\in\nz} \sum_{j\in \nz} \|H(i,j)\|_{R\otimes_\pi R} \|P_j\|_R\\
			& \leq \left(\sum_{i,j\in\nz} \|H(i,j)\|_{R\otimes_\pi R}\right) \sup_k \|P_k\|_R = \|H\|_{R,1,1} \|P\|_{R,\infty}.
	\end{align*}
A similar argument yields the bound for $\eta\in \ell^\infty(\nz, \ps\hat\otimes_R\ps^{op})$. In particular, since $\|\cdot \|_{R,\infty}\leq \|\cdot\|_{R,1}$, this gives bounded actions of $H$ on $\ell^1(\nz, \ps\hat\otimes_R\ps^{op})$ and $\ps_1^{(R)}$.
\end{proof}

\begin{prop}\label{prop:M_1_bdd}
With the action defined in the same way as above, $\mx_1(\ps\hat\otimes_R\ps^{op})\subset\bx(\ell^2(\nz, L^2(M\bar{\otimes} M^{op})))$ with $\|\cdot\|\leq \|\cdot\|_{R,1,1}$. 
\end{prop}
\begin{proof}
Since the $\ell^2$-norm is smaller than the $\ell^1$-norm, we have for $H\in \mx_1(\ps\hat\otimes_R\ps^{op})$ and $\eta=(\eta_j)\in \ell^2(\nz, L^2(M\bar\otimes M^{op}))$,
	\begin{align*}
		\left\|\left(\sum_j H(i,j) \# \eta_j\right)_i  \right\|_{\ell^2(\nz,  L^2(M\bar\otimes M^{op})}& \le \sum_{i} \left\|\sum_j H(i,j)\# \eta_j\right\|_{L^2(M\bar{\otimes} M^{op})}\\
			&\le \sum_{i, j}  \|H(i,j)\# \eta_j \|_{L^2(M\bar{\otimes} M^{op})}\\
			&\le \sum_{i,j} \| H(i,j)\|_{M\bar{\otimes} M^{op}} \|\eta_j\|_{L^2(M\bar{\otimes} M^{op})}\\
			&\le  \sum_{i,j} \| H(i,j)\|_{R\otimes_\pi R} \|\eta_j\|_{L^2(M\bar{\otimes} M^{op})}\\
			&\le \sum_i \left( \sum_j \|H(i,j )\|_{R\otimes_\pi R}^2\right)^{1/2} \left(\sum_j \|\eta_j\|^2_{L^2(M\bar{\otimes} M^{op})}\right)^{1/2}\\
			&\le \|H\|_{R,1,1} \|\eta\|_{\ell^2(\nz, L^2(M\bar{\otimes} M^{op}))}.
\end{align*}
Thus this action of $H$ defines a bounded operator on $\ell^2(\nz,L^2(M\bar\otimes M^{op}))$ with norm at most $\|H\|_{R,1,1}$.
\end{proof}

\begin{rem}\label{rem:M_1_inv_to_B_adj}
It is easily seen that for $H\in\mx_1(\ps\hat\otimes_R\ps^{op})$, $H^*$ is its adjoint in $\bx(\ell^2(\nz, L^2(M\bar{\otimes} M^{op})))$.
\end{rem}

Let $R>S>\sup_n \|X_n\|$. Using Lemma \ref{free_difference_quotients_bounded}, for $P\in \ps^{(R)}_\infty$ we define $\js P\colon \nz^2\to \ps\hat\otimes_S \ps^{op}$ by $\js P(i,j)=\bar\partial_j P_i$. The lemma (by considering $I=\nz$) implies
	\begin{align*}
		\left\|\js\colon \ps^{(R)}_\infty \to \mx_\infty(\ps\hat\otimes_S\ps^{op})\right\|&\leq \mathscr{C}(R,S)\qquad\text{and},\\
		\left\|\js\colon \ps^{(R)}_1 \to \mx_1(\ps\hat\otimes_S\ps^{op})\right\|&\leq \mathscr{C}(R,S).
	\end{align*}
Moreover, we have:

\begin{cor}\label{JD_bounded}
Let $R>S>\sup_n \|X_n\|$. Then $\js\ds$ extends to a bounded map
	\[
		\js\ds\colon \ps^{(R)}\to \mx_1(\ps\hat\otimes_S \ps^{op}),
	\]
with norm bounded by a constant $C$ depending only $R$ and $S$.
\end{cor}
\begin{proof}
For $T\in (S,R)$, we have
	\begin{align*}
		\|\js\ds P\|_{S,1,1} \leq \mathscr{C}(T,S) \|\ds P\|_{T,1}\leq \mathscr{C}(T,S) \mathscr{C}(R,T) \|P\|_R.
	\end{align*}
So take $C=\inf_{T\in (S,R)} \mathscr{C}(T,S)\mathscr{C}(R,T)$.
\end{proof}

We now seek to establish a formula which will serve as the ``dual'' of $\js$ in a certain sense (see Remark \ref{dual_sense}).

\begin{lemma}\label{Jstar_bounded}
Let $R>\sup_n\|X_n\|$. Assume that for each $n\in\nz$ the conjugate variable $\xi_n$ is given by an element of $\ps^{(R)}$, and moreover that $\sup_n \|\xi_n\|_R<\infty$. Then for $p\in\{1,\infty\}$ and $H\in \mx_p(\ps\hat\otimes_R\ps^{op})$, the expression
	\[
		\left( \sum_{j\in\nz} \partial_j^*(H(i,j))\right)_{i\in\nz}
	\]
defines an element of $\ps^{(R)}_p$ with norm at most
	\[
		\left(\frac{2}{R-\sup_n \|X_n\|} + \sup_{n\in\nz} \|\xi_n\|_R\right) \|H\|_{R,p,1}.
	\]
\end{lemma}
\begin{proof}
Fix $p\in\{1,\infty\}$ and $H\in \mx_p(\ps\hat\otimes_R\ps^{op})$. Corollary \ref{free_difference_quotient_adjoint_bounded} implies for each $i,j\in \nz$
	\[
		\|\partial_j^*(H(i,j))\|_R \leq \left(\frac{2}{R-\sup_n\|X_n\|} + \|\xi_j\|_R\right) \|H(i,j)\|_{R\otimes_\pi R}.
	\]
Then for each $i\in \nz$ we have
	\[
		\sum_{j\in\nz} \| \partial_j^*(H(i,j))\|_R\leq \left(\frac{2}{R-\sup_n \|X_n\|} + \sup_{n\in\nz} \|\xi_n\|_R\right) \sum_{j\in \nz} \|H(i,j)\|_{R\otimes_\pi R},
	\]
which yields the claimed bound when taking the supremum over $i\in\nz$ for $p=\infty$ or summing over $i\in\nz$ for $p=1$.
\end{proof}

When the conjugate variables satisfy the hypothesis of the previous lemma, we define for $H\in \mx_\infty(\ps\hat\otimes_R\ps^{op})$
	\begin{equation}\label{eqn:J*_formula}
		\js^*(H):=\left( \sum_{j\in\nz} \partial_j^*(H(i,j))\right)_{i\in\nz},
	\end{equation}
and so
	\begin{align*}
		\left\|\js^*\colon \mx_\infty(\ps\hat\otimes_R\ps^{op})\to \ps^{(R)}_\infty\right\|&\leq \frac{2}{R-\sup_n\|X_n\|} + \sup_n \|\xi_n\|_R\qquad \text{and}\\
		\left\|\js^*\colon \mx_1(\ps\hat\otimes_R\ps^{op})\to \ps^{(R)}_1\right\| &\leq \frac{2}{R-\sup_n\|X_n\|} + \sup_n \|\xi_n\|_R.
	\end{align*}

\begin{rem}\label{dual_sense}
It is possible to realize $\js^*$ as the dual map to $\js$ in the following sense. Fix $R>S>\sup_n \|X_n\|$, and let
	\[
		\js^\star\colon \mx_1(\ps\hat\otimes_S\ps^{op})^* \to \left( \ps^{(R)}_1\right)^*
	\]
be the dual map to $\js\colon \ps^{(R)}_1 \to \mx_1(\ps\hat\otimes_S\ps^{op})$. Let
	\[
		\iota_1\colon \mx_\infty(\ps\hat\otimes_S\ps^{op}) \to \mx_1(\ps\hat\otimes_S\ps^{op})^*
	\]
be the embedding defined for $H\in \mx_\infty(\ps\hat\otimes_S\ps^{op})$ by
	\[
		[\iota_1(H)](G) = \sum_{i,j\in\nz} \tau\otimes\tau^{op}( H(i,j)^*\# G(i,j)),\qquad G\in \mx_1(\ps\hat\otimes_S\ps^{op}).
	\]
Let
	\[
		\iota_2\colon \ps^{(R)}_\infty \to \left(\ps^{(R)}_1\right)^*
	\]	
be the embedding defined for $P=(P_n)_{n\in\nz}\in \ps^{(R)}_\infty$ by
	\[
		[\iota_2(P)]( F) = \sum_{n\in\nz} \tau( P_n^*F_n),\qquad F=(F_n)_{n\in\nz}\in \ps^{(R)}_1.
	\]
Then
	\[
		\js^\star\circ \iota_1 = \iota_2\circ \js^*.	
	\]
However, for our purposes it suffices to merely consider $\js^*$ as convenient notation.
\end{rem}


\subsection{Free Gibbs states}

Denote the joint law of the $X_n$, $n\in\nz$, with respect to $\tau$ by $\tau_X$.

\begin{defn}
Given $W\in \ps^{(R)}$ for $R>\sup_n \|X_n\|$, we say that $\tau_X$ is a \emph{free Gibbs state with quadratic potential perturbed by $W$} (or a \emph{free Gibbs state with perturbation $W$}) if
	\begin{equation}\label{Schwinger-Dyson_equation}
		\tau([X_n + \ds_n W] P) = \tau\otimes \tau^{op}(\partial_n P)\qquad \forall P\in \ps,\ \forall n\in\nz.
	\end{equation}
\end{defn}

\begin{rem}
Note that (\ref{Schwinger-Dyson_equation}) implies that $X$ has conjugate variables and hence $\ps^{(R)}$ and its norm $\|\cdot\|_R$ are well-defined by Corollary \ref{R_embedding}.
\end{rem}

\begin{rem}\label{V_0_def}
The terminology \emph{quadratic potential} comes from the finite sequence case, $X=(X_1,\ldots, X_N)$, considered in \cite{GS14}. The quadratic potential is
	\[
		V_0=\frac{1}{2}\sum_{i=1}^N X_i^2 \in \cz\lge X_1,\ldots, X_N\rge,
	\]
and if $V=V_0+W$, then for each $n=1,\ldots, N$
	\[
		\ds_n V = X_n + \ds_n W.
	\]
Hence equation (\ref{Schwinger-Dyson_equation}) is equivalent to
	\[
		\tau([\ds_n V]P)=\tau\otimes \tau^{op}(\partial_n P),\qquad \forall P\in \ps,
	\]
for each $n=1,\ldots, N$. In this case, $\tau$ is said to be a \emph{free Gibbs state with potential $V$}. Since $V_0\not\in\ps^{(R)}$ in the infinite sequence case, we have modified the definition to refer only to the perturbation $W\in \ps^{(R)}$.
\end{rem}

We have the following corollary to \cite[Theorem 2.1]{GMS06}.

\begin{cor}\label{Gibbs_state_unique}
Let $S=\sup_n \|X_n\|$. Fix $R> 2+S$ and let $W\in \mathscr{P}^{(R)}$. If
	\[
		\|W\|_R<e(S+1)(S+2)\log\left(\frac{R}{S+2}\right)
	\]
then there is at most one free Gibbs state with perturbation $W$ amongst normal states.
\end{cor}
\begin{proof}
Suppose two normal states $\varphi$ and $\varphi'$ on $M$ both satisfy (\ref{Schwinger-Dyson_equation}) for every $n\in \nz$. For each $d\geq 0$, define
	\begin{equation*}
		\Delta_d:=\sup_{\vec{j}\in \nz^d} |(\varphi-\varphi')(X_{j_1}\cdots X_{j_d})|.
	\end{equation*}
Note that $\Delta_d\leq 2 S^d<\infty$ for every $d\geq 0$. By normality, it suffices to show $\Delta_d(\varphi,\varphi')=0$ for all $d\geq 0$. We have immediately that $\Delta_0=0$, since both $\varphi$ and $\varphi'$ are assumed to be states.

For any $P\in \ps$ we have 
	\begin{align*}
		(\varphi-\varphi')(X_n P)= ((\varphi-\varphi')\otimes\varphi)(\partial_n P) + (\varphi'\otimes (\varphi-\varphi'))(\partial_n P) - (\varphi-\varphi')(\ds_n( W) P).
	\end{align*}
So writing $\ds_n W=\sum_{p} c_n(p) p$ as a sum of monomials $p\in \ps$, we have for $j_0=n$
	\begin{align*}
		|(\varphi-\varphi')(X_{j_0}\cdots X_{j_{d}})| &\leq 2 \sum_{k=0}^{d-1} \Delta_k S^{d-1-k} + \sum_{p}|c_{n}(p)| \Delta_{\deg(p)+d}\\
			&\leq 2 \sum_{k=0}^{d-1} \Delta_k S^{d-1-k} + \sum_{n\in \nz}\sum_{p}|c_{n}(p)| \Delta_{\deg(p)+d}.
	\end{align*}
Since the the right-hand side depends only on $d$, we see that
	\begin{align*}
		\Delta_{d+1} \leq 2 \sum_{k=0}^{d-1} \Delta_k S^{d-1-k} + \sum_{n\in \nz}\sum_{p}|c_{n}(p)| \Delta_{\deg(p)+d}.
	\end{align*}
For $\gamma>0$, set
	\begin{equation*}
		D_\gamma = \sum_{d=1}^\infty \gamma^d \Delta_d,
	\end{equation*}
which converges so long as $\gamma<\frac{1}{S}$. In the above inequality we multiply both sides by $\gamma^{d+1}$ and sum over $d\geq 0$ to obtain
	\begin{align*}
		D_\gamma&\leq \sum_{d\geq 0} 2\gamma^2\sum_{k=0}^{d-1} \gamma^k \Delta_k (\gamma S)^{d-1-k} + \sum_{d\geq 0}\gamma \sum_{n\in\nz}\sum_{p} |c_n(p)|\gamma^{-\deg(p)} \gamma^{\deg(p)+d}\Delta_{\deg(p)+d} \\
				&= 2\gamma^2\sum_{k=0}^\infty \gamma^k\Delta_k\sum_{d=k+1}^\infty (\gamma S)^{d-(k+1)} + \gamma \sum_{n\in\nz}\sum_{p} |c_n(p)|\gamma^{-\deg(p)}\sum_{d=0}^\infty \gamma^{\deg(p)+d}\Delta_{\deg(p)+d}\\
				&\leq \frac{2\gamma^2}{1-\gamma S} D_\gamma + \gamma D_\gamma \sum_{n\in\nz}\sum_{p} |c_n(p)|\gamma^{-\deg(p)}.
	\end{align*}
Let $\gamma^{-1}=S+2\in (S,R)$ so that by Lemma \ref{cyclic_derivative_bounded}, $\sum_{n\in\nz}\|\ds_n W\|_{\gamma^{-1}}\leq \mathscr{C}(R,S+2) \|W\|_R$, and hence
	\[
		\sum_{n\in\nz}\sum_p |c_n(p)| \gamma^{-\deg(p)} = \sum_{n\in\nz}\|\ds_n W \|_{\gamma^{-1}}\leq \mathscr{C}(R,S+2)\|W\|_R.
	\]
We have shown
	\[
		D_\gamma \leq \frac{2\gamma^2}{1-\gamma S}D_\gamma + \gamma D_\gamma \mathscr{C}(R,S+2)\|W\|_R = D_\gamma\frac{1+\mathscr{C}(R,S+2)\|W\|_R}{S+2}.
	\]
Thus if
	\[
		\|W\|_R< \frac{S+1}{\mathscr{C}(R,S+2)}  = e(S+1)(S+2)\log\left(\frac{R}{S+2}\right),
	\]
we arrive at the contradiction $D_\gamma<D_\gamma$ unless $D_\gamma=0$, and therefore $\Delta_d=0$ for all $d\geq 0$.
\end{proof}

\section{Construction of monotone transport}\label{Construction}

We now fix $X=(X_n)_{n\in\nz}$ to be a sequence of free semicircular random variables, so that they generate a von Neumann algebra $M$ which is isomorphic to $L(\mathbb{F}_\infty)$, the free group factor with infinitely many generators. Moreover, $M$ is equipped with a faithful normal trace $\tau$, and for each $n\in \nz$ the conjugate variable of $X_n$ with respect to $\tau$ is $X_n$ itself. Thus $\tau_X$ is the free Gibbs state with no perturbation:
	\[
		\tau(X_n P)= \tau\otimes \tau^{op}(\partial_n P)\qquad \forall P\in\ps,\ \forall n\in\nz.
	\]
Since $\sup_n \|X_n\|=\sup_n\|\xi_n\|\leq 2$, we will consider radii $R>2$.

Our goal, given $W\in \ps^{(R)}$ (for some $R>2$), is to construct a sequence $Y=(Y_n)_{n\in \nz}\subset M$ of self-adjoint operators whose joint law with respect to $\tau$ is the free Gibbs state with perturbation $W$. We shall let $\partial_{Y_n}$, $\ds_{Y_n}$, $\ds_Y$, and $\js_Y$ denote the differential operators corresponding to $Y$, so as to differentiate them from the differential operators $\partial_n$, $\ds_n$, $\ds$, and $\js$ corresponding to $X$. Thus we seek $Y$ so that
	\begin{equation}\label{SD_in_Y}
		\tau([Y_n + \ds_{Y_n} W(Y)] P) = \tau\otimes \tau^{op}(\partial_{Y_n} P)\qquad \forall P\in \ps(Y), \forall n\in \nz,
	\end{equation}
where $W(Y)$ is the power series $W$ with each $X_n$ replaced with $Y_n$ for every $n\in \nz$ (so we must require $R\ge\sup_n \|Y_n\|$ in order for $W(Y)$ to be a convergent power series).

Using that $\tau$ is a trace and $\partial_{Y_n}(P)^\dagger=\partial_{Y_n}(P^*)$, it follows that (\ref{SD_in_Y}) is equivalent to
	\[
		\lge Y_n+\ds_{Y_n} W(Y), P\rge_\tau = \lge 1\otimes 1, \partial_{Y_n} P\rge_{\tau\otimes\tau^{op}}\qquad \forall P\in \ps(Y), \forall n\in \nz,
	\]
or
	\[
		\partial_{Y_n}^*(1\otimes 1) = Y_n + \ds_{Y_n} W(Y)\qquad n\in\nz.
	\]
Then letting $1$ denote the identity element in $\mx_\infty( \ps(Y)\hat\otimes_R\ps(Y)^{op})$ (i.e. $1(i,j)=\delta_{i=j} 1\otimes 1$ for every $i,j\in\nz$), we see that (\ref{SD_in_Y}) is equivalent to
	\begin{equation}\label{vector_SD}
		\js_Y^*(1) = Y+ \ds_Y W(Y).
	\end{equation}
Similarly, we have $\js^*(1)=X$.

One way to interpret (\ref{vector_SD}), is that the conjugate variable of each $Y_n$ is a perturbation of $Y_n$ by a cyclic derivative. Since the conjugate variable of each $X_n$ is exactly $X_n$, it is not unreasonable to assume that each $Y_n$ is a perturbation of $X_n$ by a cyclic derivative. That is, $Y=X+\ds g$ for some $g\in \ps^{(R)}$. So we shall attempt to construct $Y$ of this form, and while the above may not seem like sufficient justification for this assumption, we note that in the finite sequence case \cite{GS14} showed that $Y$ always has this form for small $\|W\|_R$.

We follow the same outline as \cite[Section 3]{GS14}; in fact, the extensive analysis in Section \ref{Notation} was done precisely so that we could (as often as possible) simply appeal to the results in \cite{GS14}. We begin with a change of variables formula, which we will use to express (\ref{vector_SD}) entirely in terms of $X$.


\subsection{Equivalent forms of (\ref{vector_SD})}

Recall the formula for $\js^*$ given by (\ref{eqn:J*_formula}).

\begin{lemma}\label{change_of_variables}
Let $R>S>2$. For $Y=(Y_n)_{n\in\nz}\in \ps^{(R)}_\infty$, assume that $\js Y\in \mx_\infty(\ps\hat\otimes_S\ps^{op})$ is invertible in the Banach algebra and that $(\js Y)^*=\js Y$ where $(\js Y)^*$ is defined in \eqref{hadjoint}. Define for $j\in\nz$ and $P\in \ps$
			\[
				\hat\partial_j(P):= \sum_{i\in\nz} \partial_i(P)\#  \js Y^{-1}(i,j).
			\]
Then $\hat\partial_j$ and $\partial_{Y_j}$ agree on $\ps(Y)$. Furthermore, $\js_Y^*(1) = \js^*( (\js Y^{-1})^*)$. Here $(\js Y^{-1})^*\in \mx_\infty(\ps\hat\otimes_S \ps^{op})$ is the $*$-involution of $\js Y^{-1}$ as defined in \eqref{hadjoint}.
\end{lemma}
\begin{proof}
We first note that by Lemma \ref{free_difference_quotients_bounded}
	\begin{align*}
		\sum_{i\in\nz} \|\partial_i (P)\# \js Y^{-1}(i,j)\|_{S\otimes_\pi S} &\leq \sum_{i\in \nz} \|\partial_i P\|_{S\otimes_\pi S} \|  \js Y^{-1}(i,j) \|_{S\otimes_\pi S}\\
			& \leq \| \js Y^{-1}\|_{S,\infty,1} \sum_{i\in \nz} \|\partial_i P\|_{S\otimes_\pi S}\\
			& \leq \| \js Y^{-1}\|_{S,\infty,1} \mathscr{C}(R,S) \|P\|_R.
	\end{align*}
Thus each $\hat\partial_j$ extends to a bounded map $\hat\partial_j\colon \ps^{(R)}\to \ps\hat\otimes_S\ps^{op}$. Also, since
	\[
		\partial_i(AB)= \partial_i(A)\cdot B + A\cdot \partial_i(B) = (1\otimes B)\# \partial_i(A) + (A\otimes 1)\# \partial_i(B),
	\]
it follows that $\hat\partial_j$ satisfies the Leibniz rule. For each $k\in \nz$ we then have
	\begin{align*}
		\hat\partial_j (Y_k) = \sum_{i\in\nz} \bar\partial_i(Y_k) \# \js Y^{-1}(i,j) = \sum_{i\in \nz} \js Y(k,i)\# \js Y^{-1}(i,j) = \delta_{k=j} 1\otimes 1,
	\end{align*}
which, along with the Leibniz rule, characterizes $\partial_{Y_j}$ on $\ps(Y)$.

For each $i,j\in \nz$, define $H(i,j):=\js Y^{-1}(j,i)^*$. We claim $H=\js Y^{-1}$, in which case $(\js Y^{-1})^*=H\in \mx_\infty(\ps\hat\otimes_S\ps^{op})$. Since $(\js Y)^{-1}$ exists in $\mx_\infty(\ps\hat\otimes_S \ps^{op})$, we have
	\begin{align*}
		\sum_{j\in\nz} \js Y(i,j)H(j,k)=\sum_{j\in\nz} \js Y(j,i)^*\js Y^{-1}(k,j)^*=\sum_{j\in\nz} [\js Y^{-1}(k,j) \js Y(j,i)]^*=\de_{k=i},
	\end{align*}
where we have used $(\js Y)^*=\js Y$ to assert $\js Y(i,j)=\js Y(j,i)^*$. Thus
	\begin{align*}
		H(i,l)&=\sum_{k\in\nz} \left[\sum_{j\in\nz} \js Y^{-1}(i,j) \js Y(j,k)\right]H(k,l) \\
			&= \sum_{j\in\nz}  \js Y^{-1}(i,j) \left[\sum_{k\in\nz}  \js Y(j,k)H(k,l)\right] =(\js Y^{-1})(i,l).
	\end{align*}
Note that the change of order of summation in the above is indeed justified:
	\begin{align*}
		\sum_{k\in\nz}\sum_{j\in\nz} &\| \js Y^{-1}(i,j)\|_{S\otimes_\pi S} \|\js Y(j,k)\|_{S\otimes_\pi S} \|H(k,l)\|_{S\otimes_\pi S}\\
			&= \sum_{k\in\nz}\sum_{j\in\nz} \| \js Y^{-1}(i,j)\|_{S\otimes_\pi S} \|\js Y(j,k)\|_{S\otimes_\pi S} \|\js Y^{-1}(l,k)\|_{S\otimes_\pi S}\\
			&\leq \left(\sup_{k\in\nz} \sum_{j\in\nz} \| \js Y^{-1}(i,j)\|_{S\otimes_\pi S} \|\js Y(j,k)\|_{S\otimes_\pi S}\right) \sum_{k\in\nz}  \|\js Y^{-1}(l,k)\|_{S\otimes_\pi S}\\
			&\leq \left( \sum_{j\in \nz} \|\js Y^{-1}(i,j)\|_{S\otimes_\pi S} \sup_{j\in \nz} \sup_{k\in \nz} \|\js Y(j,k)\|_{S\otimes_\pi S}\right) \|\js Y^{-1}\|_{S,\infty,1}\\
			&\leq \|\js Y^{-1}\|_{S,\infty,1} \|\js Y\|_{S,\infty, 1}\|\js Y^{-1}\|_{S,\infty,1}.
	\end{align*}
It follows that $(\js Y^{-1})^*=H=\js Y^{-1}\in \mx_\infty(\ps\hat\otimes_S \ps^{op})$.

To show $\js_Y^*(1) = \js^*( (\js Y^{-1})^*)$, we establish the equality for each entry. Recall
	\[
		\left[\js^*\left( (\js Y^{-1})^*\right)\right]_j = \sum_{i\in\nz} \partial_i^*\left( (\js Y^{-1})^*(j,i)\right)= \sum_{i\in\nz} \partial_i^*\left( \js Y^{-1}(i,j)^*\right).
	\]
Given $P\in \ps(Y)$, we have
	\begin{align*}
		\lge 1\otimes 1, \partial_{Y_j}(P)\rge_{\tau\otimes \tau^{op}} &= \lge 1\otimes 1, \sum_{i\in \nz} \partial_i (P)\# \js Y^{-1}(i,j) \rge_{\tau\otimes\tau^{op}}\\
			&= \sum_{i\in \nz} \lge \js Y^{-1}(i,j)^*, \partial_i P\rge_{\tau\otimes\tau^{op}}\\
			&= \lge\sum_{i\in\nz} \partial_i^*\left( \js Y^{-1}(i,j)^*\right), P\rge_\tau.
	\end{align*}
Thus $1\otimes 1\in \Dom(\partial_{Y_j}^*)$ for each $j\in\nz$ with
	\[
		\partial_{Y_j}^*(1\otimes 1) = \sum_{i\in\nz} \partial_i^*\left( \js Y^{-1}(i,j)^*\right) = \left[ \js^*\left( (\js Y^{-1})^*\right)\right]_j.\qedhere
	\]
\end{proof}

The following lemma follows from the same proof used in \cite[Lemma 3.1.(iii)]{GS14}, which checked the claimed equalities entrywise.

\begin{lemma}\label{involutions_of_J}
Let $R>2$. If $g\in\ps^{(R)}$, then $(\js\ds g)^\diamond=\js\ds g$. Moreover, if $g=g^*$ then $(\js\ds g)^*=(\js\ds g)^\dagger = \js\ds g$.
\end{lemma}

\begin{cor}
Let $R>S>2$. For $g=g^*\in \ps^{(R)}$, write $f=\ds g$ and $Y=X+f$, so that $\js Y= 1+ \js f$. Assume that $1+\js f\in \mx_\infty(\ps\hat\otimes_S\ps^{op})$ is invertible. Then (\ref{vector_SD}) is equivalent to the equation
	\begin{equation}\label{vector_SD_in_X}
		\js^*\left(\frac{1}{1+\js f}\right) = X+f + (\ds W)(X+f).
	\end{equation}
\end{cor}
\begin{proof}
Lemma \ref{involutions_of_J} implies that $(\js f)^*=\js f$, and of course $1^* =1$. So $(\js Y)^*=\js Y$, and consequently $(\js Y^{-1})^*=\js Y^{-1}$. Thus Lemma \ref{change_of_variables} implies $\js_Y^*(1) = \js^*( \js Y^{-1})$, and the rest of the equivalence follows easily.
\end{proof}

The next lemma we state for the convenience of the reader since it is simply \cite[Lemma 3.3]{GS14}, the proof of which works in our present context.

\begin{lemma}
Let $R>S>2$. For $g=g^*\in \ps^{(R)}$, write $f=\ds g$ and $Y=X+f$. Assume that $1+\js f\in \mx_\infty(\ps\hat\otimes_S \ps^{op})$ is invertible. Let
	\[
		K(f) = - \js^*(\js f) - f.
	\]
Then (\ref{vector_SD_in_X}) is equivalent to
	\begin{align}\label{K_SD}
		K(f) = \ds(W(X+f)) + (\js f)\# f + (\js f)\# \js^*\left(\frac{\js f}{1+\js f}\right) - \js^*\left( \frac{(\js f)^2}{1+\js f}\right)
	\end{align}
\end{lemma}

We will see that both sides of (\ref{K_SD}) can be expressed as cyclic gradients. We begin by demonstrating this for the last two terms in (\ref{K_SD}).

\begin{lemma}\label{the_trick_lemma}
Let $R>2$. For $g\in \ps^{(R)}$, write $f=\ds g$. Then for any $m\geq -1$
	\begin{align}\label{the_trick}
		(\js f)\# \js^*\left( (\js f)^{m+1}\right) -& \js^*\left( (\js f)^{m+2}\right)\nonumber\\
				&= \frac{1}{m+2} \ds\circ(1\otimes \tau + \tau\otimes 1)\circ \Tr\left[ (\js f)^{m+2}\right].
	\end{align}
\end{lemma}
\begin{proof}
Let $S\in (2,R)$. We first note that $\js f=\js \ds g \in \mx_1(\ps\hat\otimes_S\ps^{op})$ (a Banach algebra) by Lemma \ref{JD_bounded}, and hence $(\js f)^n \in \mx_1(\ps\hat\otimes_S \ps^{op})$ for all $n\in \nz$. Then the comments following Lemma \ref{Jstar_bounded} imply $\js^*( (\js f)^n)\in \ps^{(S)}_1$ for all $n\in \nz$, and moreover $(\js f)\# \js^*( (\js f)^n)\in \ps^{(S)}_1$ by the bounded action of $\mx_1(\ps\hat\otimes_S\ps^{op})$ on $\ps^{(S)}_1$. All of this is to say that the left-hand side of (\ref{the_trick}) defines an element of $\ps^{(S)}_1$.

For the right-hand side, first note that $\Tr\left[(\js\ds g)^{m+2}\right]\in \ps\hat\otimes_T\ps^{op}$ for any $T\in (S,R)$. Then since $\tau$ is a state and $\|\cdot\|_T$ dominates the operator norm for any $T>S>2$, we know $(1\otimes \tau +\tau\otimes 1)\circ \Tr\left[(\js f)^{m+2}\right]\in \ps^{(T)}$. It then follows from Lemma \ref{cyclic_derivative_bounded} that the right-hand side also gives an element of $\ps_1^{(S)}$.

Now, we claim that for $F, G\in \ps^{(S)}_1$  if
	\[
		\sum_{n\in\nz} \tau( F_n P_n) = \sum_{n\in \nz} \tau(G_n P_n)
	\]
for any finitely supported sequence $P=(P_n)_{n\in\nz}\in \ps^{(S)}_\infty$, then $F=G$. Indeed, by letting $P\in \ps^{(S)}_\infty$ be the sequence with $(F_n - G_n)^*$ in the $n$-th entry and zeros elsewhere for some $n\in\nz$ shows $F_n=G_n$ (since $\tau$ is faithful). Letting $n$ range over $\nz$ we get $F=G$. Thus we may appeal to the proof of \cite[Lemma 3.4]{GS14}, which checks this for the two sides of (\ref{the_trick}), allowing our arsenal of bounds from Section \ref{Notation} to assuage any fears of divergence.
\end{proof}

\begin{lemma}\label{Q_Taylor_series}
Let $R>S>2$. For $g=g^*\in \ps^{(R)}$, write $f=\ds g$. Assume $\|\js f\|_{S,1,1}<1$ and let
	\[
		Q(g) =(1\otimes \tau+\tau\otimes1)\circ \Tr\left[ \js f - \log(1+\js f)\right].
	\]
 Then
	\[
		\ds Q(g) = (\js f) \# \js^*\left(\frac{\js f}{1+\js f}\right) - \js^*\left( \frac{ (\js f)^2}{1+\js f}\right).
	\]
\end{lemma}
\begin{proof}
The convergent Taylor series expansions of each side agree by Lemma \ref{the_trick_lemma}.
\end{proof}

\begin{lemma}\label{K_as_cyclic_derivative}
Let $R>2$. For $g\in \ps^{(R)}$, write $f=\ds g$. If
	\[
		K(f):= - \js^*(\js f) - f,
	\]
then
	\[
		K(f) = \ds\left\{ (1\otimes \tau + \tau\otimes 1)\circ\Tr\left[ \js \ds g\right] - \ns g\right\}.
	\]
\end{lemma}
\begin{proof}
For $m=-1$, (\ref{the_trick}) yields
	\begin{align*}
		\ds\circ(1\otimes \tau+\tau\otimes 1)\circ\Tr\left[ \js\ds g\right] &= (\js f)\#\js^*(1) - \js^*(\js f)\\
			&= (\js f)\# X - \js^*(\js f)\\
			&= \ns f - \js^*(\js f),
	\end{align*}
where $\ns$ in $\ns f$ is applied entrywise. Since $\ds$ reduces the degree by one, we have $\ns f = \ds(\ns - 1)g = \ds \ns g - f$, and the claim follows.
\end{proof}

\begin{lemma}\label{last_SD_lemma}
Let $R>S>2$. For $g=g^*\in \ps^{(R)}$, write $f\in \ds g$. Assume $\|\js f\|_{S,1,1}<1$ and let
	\[
		Q(g)=(1\otimes\tau + \tau\otimes 1)\circ \Tr\left[ \js\ds g - \log(1+ \js\ds g)\right].
	\]
Then (\ref{K_SD}) is equivalent to
	\begin{align*}
		\ds\left\{ (1\otimes \tau + \tau\otimes 1)\circ\Tr\left[\js\ds g\right] - \ns g\right\} = \ds(W(X+\ds g)) + \ds Q(g) + (\js\ds g)\#\ds g.
	\end{align*}
In particular, if $g$ satisfies
	\begin{align}\label{sufficient_SD}
		\ns g = -W(X+\ds g) - \frac{1}{2}\ds g\#\ds g + (1\otimes \tau+\tau\otimes 1)\circ\Tr\left[\log(1+\js\ds g)\right],
	\end{align}
then the joint law of $Y:=X+\ds g$ is a free Gibbs state with perturbation $W$.
\end{lemma}
\begin{proof}
The equivalence with (\ref{K_SD}) follows from Lemma \ref{K_as_cyclic_derivative} (for the left-hand side) and Lemma \ref{Q_Taylor_series} (for the right-hand side).

For the last statement in the lemma, we claim that cyclic gradient of (\ref{sufficient_SD}) is equivalent to the previous equation, in which case the joint law of $Y=X+\ds g$ will satisfy (\ref{SD_in_Y}). Indeed, it suffices to check
	\[
		\ds\left[\frac{1}{2}\ds g\# \ds g\right] = (\js \ds g)\# \ds g.
	\]
Recall $\ds_n = m\circ\diamond\circ\partial_n$ and that $\partial_n$ satisfies the Leibniz rule. Hence, we compute
	\begin{align*}
		\ds_n\left[\frac{1}{2}\sum_{i\in \nz} \ds_i(g) \ds_i(g)\right] &= \frac{1}{2}\sum_{i\in \nz} m\circ\diamond\left( \partial_n(\ds_i(g)) \cdot \ds_i(g) + \ds_i(g)\cdot \partial_n(\ds_i(g))\right)\\
			&= \frac{1}{2} \sum_{i\in \nz} m\left( [\js \ds g](i,n)^\diamond \# (\ds_i (g) \otimes 1)+ (1\otimes \ds_i(g))\# [\js\ds g](i,n)^\diamond\right)\\
			&=\sum_{i\in \nz}  [\js \ds g](i,n)^\diamond \# \ds_i (g).
	\end{align*}
Using Lemma \ref{involutions_of_J} to observe $[\js\ds g](i,n)^\diamond=[(\js\ds g)^\diamond](n,i)=[\js\ds g](n,i)$ concludes the claim and the proof.
\end{proof}


\subsection{The technical estimates}

If $g\in \ps^{(R)}$ satisfies (\ref{sufficient_SD}), then $\hat g:= \ns g$ satisfies:
	\[
		\hat g = - W(X+\ds \Sigma \hat g) - \frac{1}{2} \ds \Sigma \hat g \# \ds\Sigma\hat g + (1\otimes \tau +\tau\otimes 1)\circ \Tr\left[\log(1+\js\ds\Sigma \hat g)\right]=: F(\hat g).
	\]
Thus we must show a self-adjoint fixed point of $F$ exists, provided $W$ is sufficiently small, and towards this end we show $F$ is locally Lipschitz in a series of three lemmas (one for each component of $F(\hat g)$ above). The corollary that $F$ is locally Lipschitz is the analogue of \cite[Corollary 3.12]{GS14}; moreover, the first and second lemmas are analogues of estimates appearing in the proof of \cite[Corollary 3.12]{GS14} with slightly sharper bounds.

\begin{lemma}\label{W_Lipschitz}
Let $R>S>2$, and assume $W\in \ps^{(R+\epsilon)}$ for some $\epsilon>0$. If $g,h\in \ps^{(S)}$ satisfy $\|g\|_S,\|h\|_S\leq S\epsilon$, then
	\[
		\| W(X+\ds\Sigma g) - W(X+\ds\Sigma h)\|_S \leq \frac{\mathscr{C}(R+\epsilon,S+\epsilon)\|W\|_{R+\epsilon}}{S}\|g-h\|_S.
	\]
\end{lemma}
\begin{proof}
The condition on $g$ implies, along with Lemma \ref{cyclic_derivative_bounded}, that
	\[
		\|X+\ds\Sigma g\|_{S,\infty} \leq S+ \frac{1}{S}\|g\|_S\leq S+\epsilon,
	\]
similarly for $h$. Thus Lemmas \ref{cyclic_derivative_bounded} and \ref{mean_value_estimate} imply
	\begin{align*}
		\|W(X+\ds \Sigma g) - W(X+\ds\Sigma h)\|_S &\leq \mathscr{C}(R+\epsilon, S+\epsilon)\|W\|_{R+\epsilon}\|\ds\Sigma[g-h]\|_S\\
			&\leq \frac{\mathscr{C}(R+\epsilon,S+\epsilon)\|W\|_{R+\epsilon}}{S} \|g-h\|_S.\qedhere
	\end{align*}
\end{proof}

\begin{lemma}\label{Dot_product_Lipschitz}
Let $R>2$. Then for $g,h\in \ps^{(R)}$ we have
	\[
		\left\| \frac{1}{2} (\ds \Sigma g)\# (\ds\Sigma g) - \frac{1}{2} (\ds \Sigma h)\# (\ds \Sigma h)\right\|_R \leq \frac{\|g\|_R+\|h\|_R}{2R^2}\|g-h\|_R
	\]
\end{lemma}
\begin{proof}
We first note that by Lemma \ref{cyclic_derivative_bounded}
	\begin{align*}
		\| (\ds\Sigma g)\# (\ds\Sigma h)\|_R &\leq \sum_{n\in \nz} \| \ds_n \Sigma g \|_R \|\ds_n \Sigma h\|_R\\
			&\leq \|\ds \Sigma g\|_{R,1} \|\ds\Sigma h\|_{R,\infty}\\
			&\leq \frac{1}{R}\|g\|_R \frac{1}{R}\|h\|_R.
	\end{align*}
Applying this to the final two terms in the inequality
	\begin{align*}
		\left\| \frac{1}{2} (\ds \Sigma g)\# (\ds\Sigma g) - \frac{1}{2} (\ds \Sigma h)\# (\ds \Sigma h)\right\|_R&\\
			\leq \frac{1}{2} \| ( \ds \Sigma g)\# (\ds \Sigma [g-h])\|_R& + \frac{1}{2} \| (\ds\Sigma [g-h])\# (\ds\Sigma h)\|_R
	\end{align*}
yields the desired inequality.
\end{proof}

\begin{lemma}\label{O_Lipschitz}
Let $R>4$. For $g\in \ps^{(R)}$ satisfying $\|g\|_R< \frac{R^2}{2}$, define
	\[
		L(\Sigma g) = (1\otimes \tau+\tau\otimes 1)\circ\Tr\left[\sum_{m\geq 1} \frac{(-1)^{m+1}}{m} (\js\ds \Sigma g)^m\right].
	\]
Then the series converges in $\|\cdot\|_R$, and when $\|\js\ds\Sigma g\|_{S,1,1}<1$ for some $S\in (4,R)$ we have
	\[
		L(\Sigma g) = (1\otimes\tau+\tau\otimes 1)\circ\Tr[\log(1+\js\ds\Sigma g)].
	\]
Moveover, $L$ satisfies the local Lipschitz condition
	\[
		\|L(\Sigma g) - L(\Sigma h)\|_R \leq \| g- h\|_R \frac{2}{R^2}\left( \frac{1}{\left(1-\frac{2\|g\|_R}{R^2}\right)\left(1-\frac{2\|h\|_R}{R^2}\right)} +1\right),
	\]
and the bound
	\[
		\|L(\Sigma g)\|_R \leq \frac{2\|g\|_R}{R^2}\frac{ 2-\frac{2\|g\|_R}{R^2}}{1- \frac{2\|g\|_R}{R^2}}.
	\]
for $\|g\|_R,\|h\|_R<\frac{R^2}{2}$.
\end{lemma}
\begin{proof}
Let $g,h\in \ps^{(R)}$ satisfy $\|g\|_R, \|h\|_R<\frac{R^2}{2}$. Define
	\[
		Q(\Sigma g) = (1\otimes\tau+\tau\otimes 1)\circ\Tr[\js\ds\Sigma g]-L(\Sigma g),
	\]
(\emph{i.e.} the function considered in Lemma \ref{Q_Taylor_series}). Then we obtain the following local Lipschitz condition and bound for $Q$ from \cite[Lemma 3.8, Lemma 3.9, Corollary 3.10, and Lemma 3.11]{GS14}:
	\[
		\|Q(\Sigma g) - Q(\Sigma h)\|_R \leq \| g- h\|_R \frac{2}{R^2}\left( \frac{1}{\left(1-\frac{2\|g\|_R}{R^2}\right)\left(1-\frac{2\|h\|_R}{R^2}\right)} -1\right),
	\]
and
	\[
		\|Q(\Sigma g)\|_R \leq \frac{ (2\|g\|_R/ R^2)^2}{1- 2\|g\|_R/R^2}.
	\]
Indeed, \cite[Lemma 3.8]{GS14}, which the other cited results build off of, begins by establishing estimates on monomials and hence the same argument applies in our context. Furthermore, the recurrent hypothesis that \cite[(3.11)]{GS14} holds is satisfied in our context with $C_0=2$; that is,
	\[
		|\tau(X_{i_1}\cdots X_{i_d})|\leq 2^d,\qquad \forall \vec{i}\in \nz^d.
	\]
To obtain the claimed estimates for $L(\Sigma g)$ it then simply suffices to establish for the differing term $(1\otimes \tau + \tau \otimes 1)\circ\Tr[\js\ds\Sigma g]$ a relevant bound (since it is linear). This can be obtained by either applying \cite[Lemma 3.8]{GS14} to the case $m=1$, or using Lemma \ref{1_tensor_tau} and Lemma \ref{cyclic_derivative_bounded} to see
	\begin{align*}
		\|(1\otimes \tau + \tau \otimes 1)\circ\Tr[\js\ds\Sigma g]\|_R &\leq \frac{2}{R- 2} \sum_{n\in \nz}\|\ds_n \Sigma g\|_R\\
			&<\frac{4}{R} \frac{1}{R} \|g\|_R\\
			&= \frac{4}{R^2}\|g\|_R.
	\end{align*}
where in the second to last step we have used $R>4$.
\end{proof}

\begin{cor}\label{F_Lipschitz}
Let $R>S>4$, $\epsilon< \frac{S}{2}$, and $W\in \ps^{(R+\epsilon)}$. Then for $g,h \in \ps^{(S)}$ satisfying $\|g\|_S, \|h\|_S \leq S\epsilon$
	\[
		F(g) := - W(X+\ds \Sigma g) - \frac{1}{2} (\ds\Sigma g)\# (\ds\Sigma g) + L(\Sigma g)
	\]
defines a map into $\ps^{(S)}$ satisfying the local Lipschitz condition
	\begin{align*}
		\|F(g) - F(h)\|_S \leq \|g-h\|_S\left\{\vphantom{\left( \frac{1}{\left(1-\frac{2\|g\|_S}{S^2}\right)\left(1-\frac{2\|h\|_S}{S^2}\right)} +1\right)}\right. & \frac{\|W\|_{R+\epsilon}}{S(S+\epsilon)e\log\left(\frac{R+\epsilon}{S+\epsilon}\right)}  + \frac{\|g\|_S+\|h\|_S}{2S^2} \\
			&\qquad+\left. \frac{2}{S^2}\left( \frac{1}{\left(1-\frac{2\|g\|_S}{S^2}\right)\left(1-\frac{2\|h\|_S}{S^2}\right)} +1\right)\right\},
	\end{align*}
and the bound
	\[
		\|F(g)\|_S \leq \|g\|_S\left\{ \frac{\|W\|_{R+\epsilon}}{S(S+\epsilon)e\log\left(\frac{R+\epsilon}{S+\epsilon}\right)} + \frac{\|g\|_S}{2S^2} +\frac{2}{S^2}\left( \frac{1}{\left(1-\frac{2\|g\|_S}{S^2}\right)} +1\right)\right\} + \|W\|_S.
	\]
\end{cor}
\begin{proof}
The local Lipschitz condition of $F$ follows from Lemmas \ref{W_Lipschitz} (with the explicit formula for $\mathscr{C}(R+\epsilon,S+\epsilon)$), \ref{Dot_product_Lipschitz}, and \ref{O_Lipschitz}, where we observe that
	\[
		\|g\|_S\leq S\epsilon< \frac{S^2}{2}
	\]
(and similarly for $h$) implies that the hypotheses of Lemma \ref{O_Lipschitz} are satisfied. The bound then follows by setting $h=0$ and noting that $F(0)=-W$.
\end{proof}

\begin{cor}\label{F_has_fixed_point}
Let $R>S>4$, and $W\in \ps^{(R+1)}$. If $W(0)=0$ and
	\[
		\|W\|_{R+1} \leq \frac{e\log\left(\frac{R+1}{S+1}\right)}{2},
	\]
then $F$ maps $B:=\{g\in \ps^{(S)}\colon \|g\|_S<1\}$ into itself and satisfies $\| F(g) - F(h)\|_S< \frac{1}{2} \|g- h\|_S$ on $B$. Furthermore, there exists $\hat g\in B$ such that $F(\hat g)=\hat g$ and $\|\hat g\|_S \leq 2\|W\|_S$. If $W^*=W$, then $\hat g^* = \hat g$.
\end{cor}
\begin{proof}
The upper bound on $\|W\|_{R+1}$ implies $\mathscr{C}(R+1,S+1)\|W\|_{R+1}\leq \frac{1}{2(S+1)}$. Since $W$ has no constant term, we then have
	\begin{align*}
		\|W\|_S \leq S\frac{1}{R+1} \|W\|_{R+1} \leq S\cdot\mathscr{C}(R+1,S+1)\|W\|_{R+1}\leq\frac{1}{2}.
	\end{align*}
Then for $g\in B$, we use the bound from Corollary \ref{F_Lipschitz} (along with $S>4$) to obtain
	\[
		\|F(g)\|_S < \left\{ \frac{1}{40} + \frac{1}{32}+\frac{1}{8}\left(\frac8{7} + 1\right)\right\} + \frac{1}{2} = \frac{1}{40} + \frac{1}{32} + \frac{1}{7}+\frac{1}{8}+\frac{1}{2}<1.
	\]
Then using the Lipschitz condition for $F$ we have
	\[
		\| F(g) - F(h)\|_S \leq \|g-h\|_S\left\{\frac{1}{40} + \frac{1}{16} +\frac{1}{8}\left(\frac{64}{49} +1\right)\right\} < \frac{1}{2} \|g-h\|_S.
	\]
	
To produce the fixed point $\hat g$, we recursively define a sequence by $\hat g_n = F(\hat g_{n-1})$ and $\hat g_0=-W=F(0)$. Then
	\[
		\|\hat g_{n+1} - \hat g_n\|_S =\| F(\hat g_n) - F(\hat g_{n-1})\|_S < \frac{1}{2} \|\hat g_n - \hat g_{n-1}\|_S<\cdots <\left(\frac{1}{2}\right)^{n} \|\hat g_1 - \hat g_0\|_S.
	\]
It follows that for any $m,n\in \nz$, $m>n$, we have
	\[
		\|\hat g_m - \hat g_n\|_S \leq \sum_{k=n}^{m-1} \| \hat g_{k+1} - \hat g_k\|_S< \sum_{k=n}^\infty \left(\frac{1}{2}\right)^k \|\hat g_1 - \hat g_0\|_S=  \left(\frac{1}{2}\right)^{n-1} \|\hat g_1 - \hat g_0\|_S.
	\]
Hence $(\hat g_n)_{n\in\nz}$ is a Cauchy sequence, whose limit $\hat g$ is a fixed point of $F$. Moreover, using the Lipschitz condition of $F$ we have that
	\begin{align*}
		\|\hat g\|_S &\leq \lim_{n\to\infty} \|\hat g - \hat g_n\|_S + \|\hat g_n - \hat g_0\|_S+\|\hat g_0\|_S\\
			& \leq 2\|\hat g_1 - \hat g_0\|_S + \|\hat g_0\|_S = 2\| F(-W) - F(0)\|_S + \|W\|_S \leq 2\|W\|_S.
	\end{align*}
Lastly, note that if $W=W^*$ then $F(h^*)=F(h)^*$. Then since $\hat g_0 = -W$ is self-adjoint, each $\hat g_n = F(\hat g_{n-1})$ is self-adjoint and consequently so is $\hat g$.
\end{proof}

\begin{theorem}\label{Y_has_free_Gibbs_state}
Let $R>S>4$. Assume $W=W^*\in \ps^{(R+1)}$ satisfies $W(0)=0$ and
	\[
		\|W\|_{R+1} \leq \frac{e\log\left(\frac{R+1}{S+1}\right)}{2}.
	\]
Then there exists a cyclically symmetric $g=g^*\in \ps^{(S)}$ such that the joint law of $Y=X+\ds g$ with respect to $\tau$ is the unique free Gibbs state with perturbation $W$. Furthermore, there exists $H\in \ps^{(2.4)}_\infty$ such that $H(Y)=X$.
\end{theorem}
\begin{proof}
We let $\hat g$ be the fixed point of $F$ guaranteed by Corollary \ref{F_has_fixed_point}, and take $g=\Sigma \hat g$ (recall $\|\hat g\|_S\leq 1$). Hence $g$ satisfies (\ref{sufficient_SD}). Also, if $f=\ds g=\ds\Sigma \hat g$ then $\js f\in \mx_1(\ps\hat\otimes_{3}\ps^{op})$ with
	\[
		\|\js f\|_{3,1,1}\leq \frac{1}{3 e \log(S/3)} \|\ds\Sigma \hat g\|_S \leq \frac{1}{3 e\log(S/3)} \frac{1}{S} \|\hat g\|_S\leq  \frac{1}{12 e \log(4/3)}<1,
	\]
since $S>4$ and $\|\hat g\|_S\leq 1$. Thus Lemma \ref{last_SD_lemma} implies the joint law of $Y=X+\ds g$ is a free Gibbs state with perturbation $W$. Since $\ds = \ds \ss$, we can freely replace $g$ with the cyclically symmetric element $\ss g$.

We note that
	\[
		\frac{e\log\left(\frac{R+1}{S+1}\right)}{2} \leq 12 e\log\left(\frac{R+1}{4}\right),
	\]
so that the hypotheses of Corollary \ref{Gibbs_state_unique} are satisfied, implying that $\tau_Y$ is the unique free Gibbs state with perturbation $W$.

Finally, the constant $C$ from Remark \ref{inverse_function_constant} for $S=2.1$, $T=3$, and $R=3.9$ satisfies $C>0.3718$. Since $\|f\|_{3.9,\8} \leq \|f\|_{S,\8} \leq \frac{1}{S}\|\hat g\|_S \leq \frac{1}{4}$, it follows from Lemma \ref{inverse_function} that there exists $H\in \ps^{(2.1+C)}_{\infty} \subset \ps^{(2.4)}_{\infty}$ satisfying $H(Y)=H(X+f)=X$.
\end{proof}


\subsection{Isomorphism results}

We rephrase Theorem \ref{Y_has_free_Gibbs_state} in terms of transport.

\begin{theorem}\label{transport_thm}
Let $X=(X_n)_{n\in \nz}$ be free semicircular random variables generating the von Neumann algebra $M\cong L(\fz_\infty)$, with trace $\tau$, and let $R>S>4$. Suppose $N$ is a von Neumann algebra with a faithful normal state $\psi$, and the joint law of $Z=(Z_n)_{n\in \nz}\subset N$ with respect to $\psi$ is a free Gibbs state with perturbation $W=W^*=\ps^{(R+1)}$. If
	\[
		\|W\|_{R+1} \leq \frac{e\log\left(\frac{R+1}{S+1}\right)}{2},
	\]
then transport from $\tau_X$ to $\psi_Z$ is given by $Y=X+\ds g \in \ps^{(S)}_\infty$ for some cyclically symmetric $g=g^*\in \ps^{(S)}$. This free transport is monotone, satisfies $\|Y-X\|_{S,\infty}\to 0$ as $\|W\|_{R+1}\to 0$, and is invertible in the sense that $H(Y)=X$ for some $H\in \ps^{(2.4)}$.

In particular, there are trace-preserving isomorphisms:
	\[
		C^*(Z_n\colon n\in \nz)\cong C^*(X_n\colon n\in \nz)\qquad\text{and}\qquad W^*(Z_n\colon n\in \nz)\cong L(\fz_\infty).
	\]
\end{theorem}
\begin{proof}
We let $g$ and $H$ be as in Theorem \ref{Y_has_free_Gibbs_state}. That $Y=X+\ds g$ gives transport from $\tau_X$ to $\psi_Z$ follows from the uniqueness for the free Gibbs state with perturbation $W$. Since $g=\Sigma \hat g$ for some $\hat g\in \ps^{(S)}$, we have $Y\in \ps^{(S)}_\infty$ and
	\[
		\|Y-X\|_{S,\infty} \leq \frac{1}{S} \|\hat g\|_S \leq \frac{2}{S}\|W\|_S.
	\]

Finally, we show that the free transport is monotone. The computation in Theorem \ref{Y_has_free_Gibbs_state} showed $\|\js \ds g\|_{3,1,1} = \|\js f\|_{3,1,1}<1$. By Corollary \ref{JD_bounded}, we can find $P=P^*\in \ps$ sufficiently close to $g$ in $\ps^{(S)}$ so that $\|\js\ds P\|_{3,1,1}<1$. We view $\js\ds P\in \mx_1(\ps\hat\otimes_3\ps^{op})$ as an element of $\bx(\ell^2(\nz, L^2(M\bar{\otimes} M^{op})))$ via Proposition \ref{prop:M_1_bdd}, and in particular its operator norm is strictly less than one. By Lemma \ref{involutions_of_J}, $(\js\ds P)^*=\js\ds P$ in $\mx_1(\ps\hat\otimes_3\ps^{op})$, and so by Remark \ref{rem:M_1_inv_to_B_adj} it is self-adjoint in $\bx(\ell^2(\nz, L^2(M\bar{\otimes} M^{op})))$. All of this is to say that $1+\js \ds P >0$ in $\bx(\ell^2(\nz, L^2(M\bar{\otimes} M^{op})))$. Now, we can also guarantee that $\ds P$ is close to $\ds g=Y-X$ with respect to $\|\cdot\|_{S',1}$ for some $S' \in(4,S)$ by Lemma \ref{cyclic_derivative_bounded}, and hence with respect to the norm on $\ell^2(\nz,L^2 M)$. Thus the free transport is monotone.
\end{proof}

\section{Application to infinitely generated mixed $q$-Gaussian algebras}
In this section, we will use the free monotone transport theorem to show that the mixed $q$-Gaussian algebra $\Gamma_Q$ is isomorphic to $L(\fz_\8)$ if the structure array $Q$ has entries which are uniformly small and decay sufficiently rapidly. The analog of $C^*$-algebras is also true. As in \cites{Dab,Ne15,NZ15}, the idea is to show that the mixed $q$-Gaussian variables $X_n^Q$, $n\in\nz$, are conjugate variables to new derivations $\partial_n^{(Q)}$ that can be thought of as ``perturbations'' of the free difference quotients. From the adjoints of these new derivations one is able to obtain the conjugate variables \emph{to the free difference quotients}, which are themselves shown to be perturbations of the mixed $q$-Gaussian variables.

It is easy to see that the mixed $q$-Gaussian variables are algebraically free, hence their free difference quotients are well-defined and we denote them by $\partial_n$, $n\in\nz$. However, \emph{a priori} we do not know that $X^Q=(X_n^Q)_{n\in\nz}$ has conjugate variables, and consequently we must return to the formalism of Subsection \ref{temp_formalism} and the sequence $T=(T_n)_{n\in\nz}$ of non-commuting self-adjoint indeterminates.

Recall that $\ps(T)$ denotes the polynomials in the $T_n$, whereas $\ps(X^Q)$ denotes the polynomials in the mixed $q$-Gaussian variables. Also, when $R\geq\sup_n \|X_n^Q\|$, we have contractive maps
	\[
		\ps(T)^{(R)}\to \Gamma_Q\qquad\text{and}\qquad \ps(T)\hat\otimes_R\ps(T)^{op}\to \Gamma_Q\bar\otimes \Gamma_Q^{op},
	\]
given by replacing each $T_n$ with $X_n^Q$. The image of $P(T)\in \ps(T)^{(R)}$ (resp. $\eta(T)\in \ps(T)\hat\otimes_R\ps(T)^{op}$) under this map is denoted $P(X^Q)$ (resp. $\eta(X^Q)$).


\subsection{A natural derivation}

As in \cite{NZ15}, we consider the derivation for $j\in\nz$,
\begin{align*}
\partial_n^{(Q)}: \ps(X^Q) \to \bx(L^2(\Ga_Q)),\quad \partial_n^{(Q)}(P)=[P,r_n] =Pr_n - r_n P.
\end{align*}
Since $[l_i,r_j]=0$ for all  $i,j\in\nz$, we have
\[
\partial_j^{(Q)}(X_i^Q)(e_{k_1}\otimes \cdots \otimes e_{k_n}) = \de_{i=j}q_{ik_1}\cdots q_{ik_n} e_{k_1}\otimes \cdots \otimes e_{k_n}.
\]
For simplicity, we write $q_i(\vec{k})=q_{ik_1}\cdots q_{ik_n}$. For $d\in \nz$, we define an equivalence relation on $\nz^d$ as in \cite{NZ15}: We say $\vec{i}\sim \vec{j}$ if there exists $\si\in S_d$ such that
\[
\vec{i}= \si \cdot \vec{j} = (j_{\si(1)},..., j_{\si(d)}).
\]
Let $[\vec{i}]$ denote the equivalence class of $\vec{i}\in \nz^d$. Then we have $q_k(\vec{j}) = q_k(\vec i)$ for $\vec j\in [\vec i]$ and $k\in \nz$. For each equivalence class $[\vec i]$, we define a subspace $\fx_{[\vec i]}$ of $(\ell^2)^{\otimes n}$ by
\[
\fx_{[\vec i]}= {\rm span}\left\{ e_{j_1}\otimes\cdots\otimes e_{j_n}\colon \vec{j}\in [\vec{i}]\right\}
\]
We denote by $p_{[\vec{i}]}$ the orthogonal projection from $\fx_Q$ onto $\fx_{[\vec{i}]}$. For convenience, we write $\fx_{\emptyset}=\cz\Om$. Then it is straightforward to check that the subspaces $\fx_{[\vec{i}]}$, where $[\vec i]$ range over all equivalence classes and all $n\geq 0$, give an orthogonal decomposition of $\fx_Q$. Consequently
	\begin{align*}
		p_\Omega + \sum_{n\geq 1} \sum_{[\vec{i}]\in \nz^n/\sim} p_{[\vec{i}]} = 1,
	\end{align*}
where $p_\Omega$ is the projection onto the vacuum vector. To consider $p_\Om$ as one of the $p_{[\vec i]}$, we will often write $p_\Omega=p_{[\emptyset]} \in \nz^0/\sim$. For $n\in\nz$, we define
	\begin{align}\label{HS}
		\Xi_n =\sum_{d\geq 0} \sum_{[\vec{i}]\in \nz^d/\sim} q_n(\vec{i})p_{[\vec{i}]}.
	\end{align}
It follows that $\partial_j^{(Q)} (X_i^Q) = \de_{i=j} \Xi_j$.
Like \cites{Sh04,Dab,Ne15,NZ15}, we want to understand when $\Xi_n$ is a Hilbert--Schmidt operator. Recall that $Q_i(p)=\sum_{j\geq 1} |q_{ij}|^p$ as defined in \eqref{qip}.

\begin{lemma}
Fix $n\in \nz$ and suppose $Q_n(2)<1$. Then $\Xi_n$ is a Hilbert--Schmidt operator with $\|\Xi_n\|_{\rm HS} = (1 - Q_n(2))^{-1/2}$, where $\|\cdot\|_{\rm HS}$ denotes the Hilbert--Schmidt norm.
\end{lemma}
\begin{proof}
Fix $d\geq 0$ and denote by $\dim (S)$ the dimension of a vector space $S$. By orthogonality, we have
	\begin{align*}
		\left\|\sum_{[\vec{j}]\in \nz^d/\sim} q_n(\vec{j}) p_{[\vec{j}]} \right\|_{\rm HS}^2 &= \sum_{[\vec{j}]\in \nz^d/\sim} |q_n(\vec{j})|^2 \dim(\fx_{[\vec{j}]})\\
			& = \sum_{[\vec{j}]\in \nz^d/\sim} |q_n(\vec{j})|^2 \sum_{\vec{k}\in [\vec{j}]} 1\\
			& = \sum_{[\vec{j}]\in \nz^d/\sim} \sum_{\vec{k}\in [\vec{j}]} |q_n(\vec{k})|^2\\
			& = \sum_{\vec{k}\in \nz^d} |q_n(\vec{k})|^2 = Q_n(2)^d.
	\end{align*}
Thus when $Q_n(2)<1$ we have by orthogonality again,
\[
		\|\Xi_n\|_{\rm HS}^2 = \sum_{d\geq 0} \left\|\sum_{[\vec{j}]\in \nz^d/\sim} q_n(\vec{j}) p_{[\vec{j}]} \right\|_{\rm HS}^2 = \sum_{d\geq 0} Q_n(2)^d = \frac{1}{1-Q_n(2)}. \qedhere
\]
\end{proof}

Let $HS(\fx_Q)$ denote the space of Hilbert--Schmidt operators on $\fx_Q$. Since $HS(\fx_Q)$ is a two-sided ideal in $\bx(\fx_Q)$, if $\Xi_n$ is Hilbert--Schmidt then $\partial_n^{(Q)}$ is valued in $HS(\fx_Q)$. Recall that $L^2(\Ga_Q\bar\otimes \Ga_Q^{op})$ is isomorphic to  $HS(\fx_Q)$ via the continuation of the map
	\[
	 	\Ga_Q\otimes\Ga_Q^{op}\ni a\otimes b\mapsto \lge \cdot, b^*\Om\rge_Q a\Om.
	\]
In particular, $1\otimes 1\mapsto p_\Om$. Thus, when $\Xi_n\in HS(\fx_Q)$, we regard $\partial_n^{(Q)}$ as a densely defined derivation
	\[
		\partial_n^{(Q)}\colon L^2(\Gamma_Q)\to L^2(\Ga_Q\bar\otimes \Ga_Q^{op}),
	\]
and identify $\Xi_n\in L^2(\Ga_Q\bar\otimes\Ga_Q^{op})$. Recalling that $\partial_j^{(Q)}(X_i^Q) = \delta_{i=j} \Xi_j$, we see that under this identification we have $\partial_n^{(Q)}(\cdot) =\partial_n(\cdot)\#\Xi_n$.

\begin{prop}
If $\Xi_n\in HS(\fx_Q)$, then $1\otimes 1\in \Dom(\partial_n^{(Q)*})$ with $\partial_n^{(Q)*}(1\otimes 1) = X_n^Q$.
\end{prop}
\begin{proof}
The argument is the same as the finite generator case; see \cite{NZ15}*{Proposition 2}. The only difference in the infinite generator case is that now we need to replace the summation over $[\vec j]\in [N]^d/\sim$ by $[\vec j]\in \nz^d/\sim$. This procedure is valid because we assumed $\Xi_n\in HS(\fx_Q)$.
\end{proof}

It follows from Proposition \ref{clos_deriv} that each $\partial_{n}^{(Q)}$ is closable and for $a\otimes b\in \ps(X^Q)\otimes \ps(X^Q)^{op}$ we have
	\begin{align}\label{Q_derivation_adjiont}
		\partial_n^{(Q)*}(a\otimes b) = aX_n^Q b - m\circ(1\otimes \tau_Q\otimes 1)\circ(1\otimes \partial_n^{(Q)} + \partial_n^{(Q)}\otimes 1)(a\otimes b);
	\end{align}
see \cite{NZ15}*{Corollary 3}.


\subsection{Mixed $q$-Gaussian conjugate variables}

We first show that $\Xi_n$ is invertible under some conditions. Recall that $q_\8 = \sup_{i,j\in\nz} |q_{i,j}|$.

\begin{lemma}\label{xipi}
Fix $n\in\nz$ and $R>\sup_k\|X_k^Q\|$. Suppose $2q_\infty +(R(1-q_\8)+1)Q_n(1/2)<1$. Then there exists a noncommutative power series $\Xi_n(T)\in \ps(T)\hat\otimes_R \ps(T)^{op}$ satisfying $\Xi_n(X^Q)=\Xi_n$ and
	\begin{align*}
		\| \Xi_n(T) - 1\otimes 1 \|_{R\otimes_\pi R} \le \frac{[(R(1-q_\8)+1) Q_n(\frac12)]^2}{(1-2q_\8)^2 -[(R(1-q_\8)+1) Q_n(\frac12)]^2} =: \pi(Q, n, R).
	\end{align*}
\end{lemma}
\begin{proof}
We closely follow the argument in \cites{Dab, NZ15}. For $\vec{j}\in \nz^d$, let $\psi_{\vec{j}}(T)\in \ps(T)$ be such that $\psi_{\vec{j}}:=\psi_{\vec{j}}(X^Q)$ is the Wick word corresponding to $e_{\vec j}:=e_{j_1}\otimes\cdots \otimes e_{j_d}$. Namely,
\[
\psi_{\vec j} = W(e_{j_1}\otimes\cdots \otimes e_{j_d});
\]
see \cite{NZ15}*{(5)}. In view of \eqref{HS}, we want to write each $p_{[\vec j]}$ as a sum of tensor products. Note that $\fx_{[\vec j]}$ is finite dimensional, and that $\lge e_{\vec i},e_{\vec j}\rge_Q$ equals zero unless $\vec i \sim \vec j$. Let $G_{[\vec j]}$ denote the Gram matrix of the natural basis $(e_{i_1}\otimes\cdots\otimes e_{i_d})$ of $\fx_{[\vec j]}$ in the inner product $\lge\cdot, \cdot\rge_Q$. In other words, $G_{[j]}$ is the matrix of $P^{(d)}$ restricted to $\fx_{[j]}$, where $P^{(d)}$ is defined in \eqref{pninn}. Since $P^{(d)}$ is strictly positive, we may define $B([\vec{j}]) = G_{[\vec j]}^{-1/2}$. For $\vec i \in [\vec j]$, let
\[
f_{\vec i}(T)= \sum_{\vec k \in [\vec j]} B([\vec j])_{\vec i \vec k}\psi_{\vec k}(T).
\]
It is straightforward to check that $\{f_{\vec{i}}(X^Q)\Om \}_{\vec i\in [\vec j ]}$ is an orthonormal basis of $\fx_{[\vec j]}$. Let us define
\begin{align*}
p_{[\vec j]}(T)& = \sum_{\vec i\in[\vec j]} f_i(T)\otimes f_i(T)^*\\
&= \sum_{\vec i, \vec k, \vec l\in [\vec j]} B([\vec j])_{\vec i \vec k} B([\vec j])_{\vec l \vec i} \psi_{\vec k}(T) \otimes \psi_{\vec l}(T)^* = \sum_{\vec{k},\vec{l}\in [\vec{j}]} (B([\vec j])^2)_{\vec{l}\vec{k}} \psi_{\vec{k}}(T)\otimes \psi_{\vec{l}}(T)^*.
\end{align*}
Under the identification of $HS(\fx_Q)$ and $L^2(\Ga_Q\bar\otimes \Ga_Q^{op})$, we have
	\begin{align*}
	p_{[\vec j]} = p_{[\vec j]}(X^Q)	 = \sum_{\vec{k},\vec{l}\in [\vec{j}]} (B([\vec j])^2)_{\vec{l}\vec{k}} \psi_{\vec{k}}\otimes \psi_{\vec{l}}^*,
	\end{align*}
and now $\Xi_n$ can be written as a sum of tensors. We consider
\[
	\Xi_n(T) :=  \sum_{d\ge 0} \sum_{[\vec j ] \in \nz^d/\sim} q_{n}(\vec j) p_{[\vec j]} (T) = \sum_{d\ge 0} \sum_{[\vec j ] \in \nz^d/\sim} q_{n}(\vec j) \sum_{\vec{k},\vec{l}\in [\vec{j}]} (B([\vec j])^2)_{\vec{l}\vec{k}} \psi_{\vec{k}}(T)\otimes \psi_{\vec{l}}(T)^*.
\]
As in the proof of \cite{Dab}*{Corollary 29}, we have for $\vec{j}\in \nz^d$
	\begin{align*}
		\sup_{ |\vec j|=d}\|\psi_{\vec{j}}(T)\|_R \leq \left( R+ \frac{1}{1-q_\infty}\right)^d,
	\end{align*}
Also, $P^{(n)}$ is block diagonal so entrywise we have the bound for $\vec{k}, \vec{l}\in [\vec j]$ from \eqref{pnorm}
	\begin{align*}
		| (B([\vec j])^2)_{\vec{l}\vec{k}}| \le \|G_{[\vec j]}^{-1}\|\le \|(P^{(d)})^{-1}\| \leq \left( \frac{1-q_\infty}{1-2q_\infty}\right)^d,
	\end{align*}
where we note that $q_\8<\frac12$ is implied by the hypothesis. Then we estimate
	\begin{align*}
		\left\|\sum_{[\vec{j}]\in \nz^d/\sim} q_n(\vec{j})\right. &\left. \sum_{\vec{k},\vec{l}\in [\vec{j}]} (B([\vec j])^2)_{\vec{l}\vec{k}} \psi_{\vec{k}}(T)\otimes \psi_{\vec{l}}(T)^* \right\|_{R\otimes_\pi R}\\
			&\leq \sum_{[\vec{j}]\in \nz^d/\sim} |q_n(\vec{j})| \sum_{\vec{k}, \vec{l}\in[\vec{j}]} | (B([\vec j])^2)_{\vec{l}\vec{k}}| \left( R+ \frac{1}{1-q_\infty}\right)^{2d}\\
			&\le  \left( R+ \frac{1}{1-q_\infty}\right)^{2d} \left( \frac{1-q_\infty}{1-2q_\infty}\right)^d \sum_{[\vec{j}]\in \nz^d/\sim} \sum_{\vec{k}\in [\vec{j}]} \sqrt{|q_n(\vec{k})|} \sum_{\vec{l}\in [\vec{j}]} \sqrt{|q_n(\vec{l})|}\\
			&\leq \left( R+ \frac{1}{1-q_\infty}\right)^{2d} \left( \frac{1-q_\infty}{1-2q_\infty}\right)^d \sum_{\vec{k}\in \nz^d} \sqrt{|q_n(\vec{k})|}  \sum_{\vec{l}\in \nz^d} \sqrt{|q_n(\vec{l})|}\\
			&\le \left[\frac{R(1-q_\8)+1}{1-2q_\8}Q_n\left(1/2\right) \right]^{2d}.
	\end{align*}
The hypothesis implies this quantity is less than one, and therefore summing over $d\geq 1$ we have convergence according to a geometric series.
\end{proof}

\begin{rem}\label{one_hyp_to_rule_them_all}
  If $0<\pi(Q,n,R)<1$ for all $n\in\nz$, then $q_\8<\frac12$, $Q_n(2)<1$, and $2q_\infty +(R(1-q_\8)+1)Q_n(\frac12)<1$ for all $n\in\nz$. To see this, we claim that
  \[
  1-2q_\8 + (R-Rq_\8+1)Q_n(\frac12)>0, \quad \text{for some  } n.
  \]
  Indeed, suppose this is not true. Then $Q_n(\frac12)-q_\8<0$ for all $n\in\nz$. But $\sup_{n\in\nz} Q_n(\frac12)>\sup_{i,j} |q_{ij}|=q_\8$, a contradiction. From the claim and $\pi(Q,n,R)>0$, we deduce $1-2q_\8 - (R-Rq_\8+1)Q_n(\frac12)>0$ for some $n$. It follows then $q_\8<\frac12$. This yields finally $2q_\infty +(R(1-q_\8)+1)Q_n(\frac12)<1$ and $Q_n(2)\le Q_n(\frac12)<1$ for all $n\in\nz$.
\end{rem}

If $\pi(Q,n,R)<1$, Lemma \ref{xipi} implies that $\Xi_n(T)$ is invertible in $\ps(T)\hat\otimes_R\ps(T)^{op}$ with inverse $\Xi_n^{-1}(T)$. Since $\|\cdot\|_{R\otimes_\pi R}$ dominates the operator norm $\|\cdot\|_{\Gamma_Q\bar{\otimes}\Gamma_Q}$, it follows that $\Xi_n$ is then invertible in $\Ga_Q\bar\otimes\Ga_Q^{op}$ with $\Xi_n^{-1}=\Xi_n^{-1}(X^Q)$. We will obtain the conjugate variables to the free difference quotients as the image of $(\Xi_n^{-1})^*$ under the adjoint $\partial_n^{(Q)*}$, but will need to show $(\Xi_n^{-1})^*\in \Dom(\partial_n^{(Q)*})$.

\begin{prop}\label{conju_exist}
Let $R>\sup_n \|X_n^Q\|$. Suppose $0<\pi(Q,n, R)< 1$ for all $n\in\nz$. Then we have
\begin{enumerate}
\item[(i)] There exist self-adjoint noncommutative power series $\xi_n(T)\in \ps(T)^{(R)}$ for each $n\in \nz$ such that $\{\xi_n(X^Q)\}_{n\in\nz }$ are the conjugate variables of $X^Q$.
\item[(ii)] For each $n\in \nz$,
	\[
		\|\xi_n(T) - T_n\|_R \leq \left(R+\frac{2\|\Xi_n(T)\|_{R\otimes_\pi R}}{R-\sup_k \|X_k^Q\|}\right) \frac{\pi(Q, n,R)}{1-\pi(Q,n,R)},
	\]
which tends to zero as $Q_n(1/2)\to 0$.
\end{enumerate}
\end{prop}
\begin{proof}
We first claim that if $\eta\in \ps(T)\otimes\ps(T)^{op}$, then $\partial_n^{(Q)*}(\eta(X^Q))=P(X^Q)$ for some $P\in \ps(T)^{(R)}$ such that
	\[
		\|\partial_n^{(Q)*}(\eta(X^Q))\|\leq \| P\|_R \leq \left(R+ \frac{2\|\Xi_n(T)\|_{R\otimes_\pi R}}{R- \sup_n\|X_n^Q\|}\right)\|\eta\|_{R\otimes_\pi R}.
	\]
Let $\partial_{T_n}^{(Q)}(\cdot)=\partial_{T_n}(\cdot)\# \Xi_n(T)$, and let $\tau_{X^Q}\colon \ps(T)\to \cz$ denote the joint law of $X^Q$ with respect to $\tau_Q$. Define
	\begin{align}\label{formal_Q_adjoint}
		P:=\eta\# T_n - m\circ (1\otimes \tau_{X^Q}\otimes 1)\circ(1\otimes\partial_{T_n}^{(Q)}+\partial_{T_n}^{(Q)}\otimes 1)(\eta),
	\end{align}
so that by (\ref{Q_derivation_adjiont}) $P(X^Q)=\partial_n^{(Q)^*}(\eta(X^Q))$. Now, using the same argument as in Lemma \ref{1_tensor_tau} it is easy to see that for $A\in \ps(T)$ and $\zeta\in \ps(T)\otimes \ps(T)^{op}$
	\[
		\|(1\otimes \tau_{X^Q})(\partial_{T_n}(A)\#\zeta)\|_R \leq \frac{1}{R-\sup_n \|X_n^Q\|} \|A\|_R\|\zeta\|_{R\otimes_\pi R}.
	\]
Thus this bound extends to $A\in \ps(T)^{(R)}$ and $\zeta\in \ps(T)\hat\otimes_R\ps(T)^{op}$. In particular, we have
	\[
		\|(1\otimes \tau_{X^Q})(\partial_{T_n}^{(Q)} A)\|_R \leq \frac{\|\Xi_n(T)\|_{R\otimes_\pi R}}{R-\sup_n \|X_n^Q\|}\|A\|_R.
	\]
The same estimate holds for $(\tau_{X^Q}\otimes 1)\circ\partial_{T_n}^{(Q)}$, and so proceeding as in Lemma \ref{free_difference_quotient_adjoint_bounded} we obtain the claimed upper bound for $\|P\|_R$. The lower bound holds simply because the operator norm in $\Ga_Q$ is dominated by $\|\cdot\|_R$.

We next claim that for each $n\in \nz$, we have $(\Xi_n^{-1})^*\in \Dom(\partial_n^{(Q)*})$. Indeed, recall that $\ps(T)\otimes\ps(T)^{op}$ is dense in $\ps(T)\hat\otimes_R\ps(T)^{op}$. Then by Lemma \ref{xipi}, $\pi(Q,n,R)<1$ implies that $\Xi_n^{-1}\in \Ga_Q\bar\otimes\Ga_Q^{op}$ and there is a sequence $(\eta_k)_{k\in\nz}\subset \ps(T)\otimes\ps(T)^{op}$ such that
	\[
		\|\Xi_n^{-1} - \eta_k(X^Q)\|\leq \|\Xi_n^{-1}(T) - \eta_k\|_{R\otimes_\pi R} \stackrel{k\to\8}{\longrightarrow} 0.
	\]
Since $*$ is an isometry in both $\Ga_Q\bar\otimes \Ga_Q^{op}$ and $\ps(T)\hat\otimes_R\ps(T)^{op}$, we can assume that $(\eta_k(X^Q))_{k\in \nz}$ actually approximates $(\Xi_n^{-1})^*$. By the previous claim we then have a sequence $(P_k)_{k\in\nz}\subset \ps(T)^{(R)}$ such that
	\[
		\|\partial_n^{(Q)^*}(\eta_k(X^Q)) - \partial_n^{(Q)^*}(\eta_l(X^Q))\|\leq \|P_k - P_l\|_R \leq \left(R+ \frac{2\|\Xi_n(T)\|_{R\otimes_\pi R}}{R- \sup_n\|X_n^Q\|}\right)\| \eta_k - \eta_l\|_{R\otimes_\pi R}.
	\]
Thus $\left( \partial_n^{(Q)^*}(\eta_k(X^Q))\right)_{k\in \nz}$ converges in $\Ga_Q$, which implies $(\Xi_n^{-1})^*\in \Dom(\partial_n^{(Q)*})$ since $\partial_n^{(Q)*}$ is closed, and $(P_k)_{k\in \nz}$ converges to some $\xi_n(T)\in \ps(T)^{(R)}$.

Recall that $\partial_n^{(Q)}(\cdot) = \partial_n(\cdot)\# \Xi_n$, or $\partial_n(\cdot)= \partial_n^{(Q)}(\cdot)\# \Xi_n^{-1}$. For $P\in \ps(X^Q)$ we have
	\[
		\lge\xi_n(X^Q), P\rge_{\tau_Q} = \lge (\Xi_n^{-1})^*, \partial_n^{(Q)} P\rge_{\tau_Q\otimes\tau_Q^{op}} = \lge 1\otimes 1, \partial_n(P)\rge_{\tau_Q\otimes\tau_Q^{op}}.
	\]
Thus $\{\xi_n(X^Q)\}_{n\in \nz}$ are the conjugate variables.

One can verify that $\xi_n(T)$ is self-adjoint directly from (\ref{formal_Q_adjoint}) using that $\Xi_n(T)^\dagger=\Xi_n(T)$ (and thus $(\Xi_n^{-1}(T)^*)^\dagger = \Xi_n^{-1}(T)^*$). Alternatively, one can check that $\left[\partial_n^{(Q)^*}(\eta)\right]^*=\partial_n^{(Q)^*}(\eta^\dagger)$ (either weakly or using (\ref{Q_derivation_adjiont})). Then $\Xi_n^\dagger=\Xi_n$ implies $\xi_n(X^Q)=\partial_n^{(Q)^*}([\Xi_n^{-1}]^*)$ is self-adjoint. Since we now know $X^Q$ has conjugate variables, $\xi_n(X^Q)$ being self-adjoint implies $\xi_n(T)$ is self-adjoint since the ``evaluation in $(X_k^Q)_{k\in \nz}$'' map is an embedding by Corollary \ref{R_embedding}.

To see (ii), we recall that $\partial_n^{(Q)*}(1\otimes 1) = X_n^Q = T_n(X^Q)$. So by the initial claim
	\[
		\|\xi_n(T) - T_n\|_R \leq \left(R+ \frac{2\|\Xi_n(T)\|_{R\otimes_\pi R}}{R- \sup_n\|X_n^Q\|}\right)\| \Xi_n^{-1}(T)^* - 1\otimes 1\|_{R\otimes_\pi R}.
	\]
We can then bound $\|\Xi_n^{-1}(T)^* -1\otimes 1\|_{R\otimes_\pi R}=\|\Xi_n^{-1}(T) - 1\otimes 1\|_{R\otimes_\pi R}$ using Lemma \ref{xipi}.
\end{proof}


\subsection{Isomorphism of $\Gamma_Q$ and $L(\fz_\8)$}

We use Proposition \ref{conju_exist} to show that Theorem \ref{transport_thm} can be applied to $\Gamma_Q$ provided the structure array $Q$ satisfies a certain ``smallness'' condition.

\begin{theorem}\label{perturbation_formula}
Let $R>\sup_n \|X_n^Q\|$. Suppose $0< \pi(Q,n,R)<1$ for all $n\in\nz$, and that
	\[
		\sum_{n\in\nz}\frac{\pi(Q,n,R)}{1-\pi(Q,n,R)}<\infty.
	\]
With $\xi_n = \xi_n(X^Q)$ as in Proposition \ref{conju_exist}, let
	\begin{align*}
		W=\Sigma\left( \frac{1}{2}\sum_{n\in \nz} X_n^Q(\xi_n -X_n^Q) + (\xi_n - X_n^Q)X_n^Q\right).
	\end{align*}
Then	$W=W^*$, $\ds_n W = \xi_n - X_n^Q$ for each $n\in \nz$,
	\[
		\|W\|_R \leq \frac{1}{2} R\left( R+ \frac{4}{R-\sup_n \|X_n^Q\|}\right)\sum_{n\in \nz} \frac{\pi(Q,n,R)}{1-\pi(Q,n,R)}<\infty,
	\]
and the joint law of the $X_n^Q$ with respect to $\tau_Q$ is a free Gibbs state with perturbation $W$.
\end{theorem}
\begin{proof}
Since $\xi_n - X_n^Q$ has no constant term, $\Sigma$ contributes a factor of at most $\frac12$ to each monomial in $W$. Hence
	\[
		\|W\|_R \leq \frac{1}{2} \sum_{n\in \nz}  R \|\xi_n - X_n^Q\|_R,
	\]
and the rest of the bound follows from Proposition \ref{conju_exist} and the observations that $\pi(Q,n,R)<1$ implies $\|\Xi_n(T)\|_{R\otimes_\pi R}\leq 2$. Each $\xi_n$ being self-adjoint implies $W$ is self-adjoint. To see $\ds_n W= \xi_n -X_n^Q$, we proceed as in Step 4 of the proof of \cite[Theorem 34]{Dab}. First note that $\ds_n W = \xi_n - X_n^Q$ is equivalent to
	\[
		\ds_n (\ns W) = (1+ \ns)(\ds_n W) = \xi_n - 2X_n^Q +\ns \xi_n = \xi_n -2X_n^Q+ \sum_{j\in \nz} \partial_j(\xi_n)\# X_j^Q.
	\]
Now, since $\ds_n(AB+BA)= 2\ds_n(AB)$, we have
	\begin{align*}
		\ds_n(\ns W) &= \ds_n\left(\sum_{j\in \nz} X_j^Q\xi_j - (X_j^Q)^2\right)\\
				&= \sum_{j\in \nz} \ds_n(X_j^Q \xi_j) - 2X_n^Q.
	\end{align*}
Recall that $\ds_n(AB) = \partial_n(A)^\diamond \# B + \partial_n(B)^\diamond\# A$, since $\ds_n=m\circ\diamond\circ\partial_n$. Thus
	\begin{align*}
		\ds_n(\ns W) = \xi_n - 2X_n^Q + \sum_{j\in \nz} \partial_n(\xi_j)^\diamond \#X_j^Q.
	\end{align*}
From \cite[Lemma 36]{Dab} we have $\partial_n(\xi_j)^\diamond = \partial_j(\xi_n)$ and the desired equality follows.
\end{proof}

From (\ref{lnorm}), we have that
	\begin{align*}
		\|X_n^Q\|\leq \left\{\begin{array}{cl} \frac{2}{\sqrt{1-q_{nn}}} & \text{if }q_{nn}\in [0,1),\\
									2 & \text{if }q_{nn}\in (-1,0]. \end{array}\right.
	\end{align*}
If we assume $0<\pi(Q,n,R)<1$ for all $n\in \nz$ (so that $q_\8<\frac12$ by Remark \ref{one_hyp_to_rule_them_all}), then $\sup_n\|X_n^Q\|\leq 2\sqrt{2}<4$.

\begin{cor}\label{mixed_q_isom}
Let $R>5$. If the structure array $Q$ for the mixed $q$-Gaussian algebra $\Gamma_Q$ satisfies $0<\pi(Q,n,R)<1$ for all $n\in \nz$, and
	\[
		\sum_{n\in \nz} \frac{\pi(Q,n,R)}{1-\pi(Q,n,R)} < \frac{e\log\left(\frac{R}{5}\right)}{R\left( R+ \frac{4}{R-\sup_n \|X_n^Q\|}\right)},
	\]
then $\Gamma_Q\cong L(\fz_\8)$ and $C^*(X_n^Q\colon n\in \nz)\cong C^*(X_n\colon n\in \nz)$, where $\{X_n\}_{n\in \nz}$ is a free semicircular family.
\end{cor}
\begin{proof}
Let $W$ be as in Theorem \ref{perturbation_formula}. These conditions imply
	\[
		\|W\|_R< \frac{e \log\left(\frac{R}{5}\right)}{2}.
	\]
Hence $\exists S\in (5,R)$ such that
	\[
		\|W\|_R= \frac{e \log\left(\frac{R}{S}\right)}{2},
	\]
and the isomorphisms follow from Theorem \ref{transport_thm}.
\end{proof}

\begin{rem}
Roughly speaking the condition on $\sum_{n\in\nz}\frac{\pi(Q,n,R)}{1-\pi(Q,n,R)}$ is reasonable since it implies that that $X_i^Q$ and $X_j^Q$ are ``almost free'' when $i+j$ is large and that each $X_n^Q$ is close in distribution to a semicircular variable. Indeed, for $\sum_{n\in\nz}\frac{\pi(Q,n,R)}{1-\pi(Q,n,R)}$ to be finite we must have $|q_{ij}|\to 0$ sufficiently fast as $i+j$ grows, and $q_{ij}$ encodes the relative freeness of $X_i^Q$ and $X_j^Q$. For $\sum_{n\in\nz}\frac{\pi(Q,n,R)}{1-\pi(Q,n,R)}$ to further satisfy the necessary inequality $q_\8$ must be sufficiently small. Since each $X_n^Q$ is a $q_{nn}$-semicircular variable, $q_\8$ being small implies that $X_n^Q$ is close in distribution to a semicircular variable.
\end{rem}

\begin{rem}
As noted in the introduction, in the case $q_{ij}=q^{i+j-1}$ for some $-1<q<1$, if $|q|<0.0002488$ then the hypotheses of Corollary \ref{mixed_q_isom} hold for $R=6.7$.
\end{rem}

\section*{Acknowledgements}
We are grateful to Dan Voiculescu for his timely asking of Question \ref{ques1} with the mixed $q$-Gaussian algebras in mind. Although this question was asked many times, it is in the context of mixed $q$-Gaussian algebras that we see some hope for the first time, which eventually resulted in this paper. We would also like to thank Marek Bo\.zejko for sending us his numerical evidence, and Marius Junge and Dimitri Shlyakhtenko for their comments and encouragement.

\bibliographystyle{alpha}
\bibliography{qisom}

\begin{bibdiv}
\begin{biblist}

\bib{Av}{article}{
      author={{Avsec}, S.},
       title={{Strong Solidity of the q-Gaussian Algebras for all $-1 < q <
  1$}},
        date={2011-10},
     journal={ArXiv e-prints},
      eprint={1110.4918},
}

\bib{BKS}{article}{
      author={Bo{\.z}ejko, Marek},
      author={K{\"u}mmerer, Burkhard},
      author={Speicher, Roland},
       title={{$q$}-{G}aussian processes: non-commutative and classical
  aspects},
        date={1997},
        ISSN={0010-3616},
     journal={Comm. Math. Phys.},
      volume={185},
      number={1},
       pages={129\ndash 154},
         url={http://dx.doi.org/10.1007/s002200050084},
      review={\MR{1463036 (98h:81053)}},
}

\bib{Boz}{incollection}{
      author={Bo{\.z}ejko, Marek},
       title={Completely positive maps on {C}oxeter groups and the
  ultracontractivity of the {$q$}-{O}rnstein-{U}hlenbeck semigroup},
        date={1998},
   booktitle={Quantum probability ({G}da\'nsk, 1997)},
      series={Banach Center Publ.},
      volume={43},
   publisher={Polish Acad. Sci., Warsaw},
       pages={87\ndash 93},
      review={\MR{1649711 (2000k:20053)}},
}

\bib{Bre91}{article}{
      author={Brenier, Yann},
       title={Polar factorization and monotone rearrangement of vector-valued
  functions},
        date={1991},
        ISSN={0010-3640},
     journal={Comm. Pure Appl. Math.},
      volume={44},
      number={4},
       pages={375\ndash 417},
         url={http://dx.doi.org/10.1002/cpa.3160440402},
      review={\MR{1100809 (92d:46088)}},
}

\bib{BS91}{article}{
      author={Bo{\.z}ejko, Marek},
      author={Speicher, Roland},
       title={An example of a generalized {B}rownian motion},
        date={1991},
        ISSN={0010-3616},
     journal={Comm. Math. Phys.},
      volume={137},
      number={3},
       pages={519\ndash 531},
         url={http://projecteuclid.org/euclid.cmp/1104202738},
      review={\MR{1105428 (92m:46096)}},
}

\bib{BS94}{article}{
      author={Bo{\.z}ejko, Marek},
      author={Speicher, Roland},
       title={Completely positive maps on {C}oxeter groups, deformed
  commutation relations, and operator spaces},
        date={1994},
        ISSN={0025-5831},
     journal={Math. Ann.},
      volume={300},
      number={1},
       pages={97\ndash 120},
         url={http://dx.doi.org/10.1007/BF01450478},
      review={\MR{1289833 (95g:46105)}},
}

\bib{Con}{article}{
      author={Connes, A.},
       title={Classification of injective factors. {C}ases {$II_{1},$}
  {$II_{\infty },$} {$III_{\lambda },$} {$\lambda \not=1$}},
        date={1976},
        ISSN={0003-486X},
     journal={Ann. of Math. (2)},
      volume={104},
      number={1},
       pages={73\ndash 115},
      review={\MR{0454659 (56 \#12908)}},
}

\bib{Dab}{article}{
      author={Dabrowski, Yoann},
       title={A free stochastic partial differential equation},
        date={2014},
        ISSN={0246-0203},
     journal={Ann. Inst. Henri Poincar\'e Probab. Stat.},
      volume={50},
      number={4},
       pages={1404\ndash 1455},
         url={http://dx.doi.org/10.1214/13-AIHP548},
      review={\MR{3270000}},
}

\bib{GMS06}{article}{
      author={Guionnet, Alice},
      author={Maurel-Segala, Edouard},
       title={Combinatorial aspects of matrix models},
        date={2006},
        ISSN={1980-0436},
     journal={ALEA Lat. Am. J. Probab. Math. Stat.},
      volume={1},
       pages={241\ndash 279},
      review={\MR{2249657 (2007g:05087)}},
}

\bib{GS14}{article}{
      author={Guionnet, A.},
      author={Shlyakhtenko, D.},
       title={Free monotone transport},
        date={2014},
        ISSN={0020-9910},
     journal={Invent. Math.},
      volume={197},
      number={3},
       pages={613\ndash 661},
         url={http://dx.doi.org/10.1007/s00222-013-0493-9},
      review={\MR{3251831}},
}

\bib{JZ15}{article}{
      author={{Junge}, M.},
      author={{Zeng}, Q.},
       title={{Mixed $q$-Gaussian algebras}},
        date={2015-05},
     journal={ArXiv e-prints},
      eprint={1505.07852},
}

\bib{KN11}{article}{
      author={Kennedy, Matthew},
      author={Nica, Alexandru},
       title={Exactness of the {F}ock space representation of the
  {$q$}-commutation relations},
        date={2011},
        ISSN={0010-3616},
     journal={Comm. Math. Phys.},
      volume={308},
      number={1},
       pages={115\ndash 132},
         url={http://dx.doi.org/10.1007/s00220-011-1323-9},
      review={\MR{2842972}},
}

\bib{Kr00}{article}{
      author={Kr{\.o}lak, Ilona},
       title={Wick product for commutation relations connected with
  {Y}ang-{B}axter operators and new constructions of factors},
        date={2000},
        ISSN={0010-3616},
     journal={Comm. Math. Phys.},
      volume={210},
      number={3},
       pages={685\ndash 701},
         url={http://dx.doi.org/10.1007/s002200050796},
      review={\MR{1777345 (2001i:46103)}},
}

\bib{Kr05}{article}{
      author={Kr{\'o}lak, Ilona},
       title={Contractivity properties of {O}rnstein-{U}hlenbeck semigroup for
  general commutation relations},
        date={2005},
        ISSN={0025-5874},
     journal={Math. Z.},
      volume={250},
      number={4},
       pages={915\ndash 937},
         url={http://dx.doi.org/10.1007/s00209-005-0801-1},
      review={\MR{2180382 (2006i:81105)}},
}

\bib{LP99}{article}{
      author={Lust-Piquard, Fran{\c{c}}oise},
       title={Riesz transforms on deformed {F}ock spaces},
        date={1999},
        ISSN={0010-3616},
     journal={Comm. Math. Phys.},
      volume={205},
      number={3},
       pages={519\ndash 549},
         url={http://dx.doi.org/10.1007/s002200050688},
      review={\MR{1711277 (2001j:46104)}},
}

\bib{MSW17}{article}{
      author={Mai, Tobias},
      author={Speicher, Roland},
      author={Weber, Moritz},
       title={Absence of algebraic relations and of zero divisors under the
  assumption of full non-microstates free entropy dimension},
        date={2017},
        ISSN={0001-8708},
     journal={Adv. Math.},
      volume={304},
       pages={1080\ndash 1107},
         url={http://dx.doi.org/10.1016/j.aim.2016.09.018},
      review={\MR{3558227}},
}

\bib{Ne15}{article}{
      author={Nelson, Brent},
       title={Free monotone transport without a trace},
        date={2015},
        ISSN={0010-3616},
     journal={Comm. Math. Phys.},
      volume={334},
      number={3},
       pages={1245\ndash 1298},
         url={http://dx.doi.org/10.1007/s00220-014-2148-0},
      review={\MR{3312436}},
}

\bib{Nou}{article}{
      author={Nou, Alexandre},
       title={Non injectivity of the {$q$}-deformed von {N}eumann algebra},
        date={2004},
        ISSN={0025-5831},
     journal={Math. Ann.},
      volume={330},
      number={1},
       pages={17\ndash 38},
         url={http://dx.doi.org/10.1007/s00208-004-0523-4},
      review={\MR{2091676 (2005k:46187)}},
}

\bib{NZ15}{article}{
      author={{Nelson}, B.},
      author={{Zeng}, Q.},
       title={{An application of free transport to mixed \$q\$-Gaussian
  algebras}},
        date={2015-07},
     journal={ArXiv e-prints},
      eprint={1507.04824},
}

\bib{Ri05}{article}{
      author={Ricard, {\'E}ric},
       title={Factoriality of {$q$}-{G}aussian von {N}eumann algebras},
        date={2005},
        ISSN={0010-3616},
     journal={Comm. Math. Phys.},
      volume={257},
      number={3},
       pages={659\ndash 665},
         url={http://dx.doi.org/10.1007/s00220-004-1266-5},
      review={\MR{2164947 (2006e:46069)}},
}

\bib{Sh04}{article}{
      author={Shlyakhtenko, Dimitri},
       title={Some estimates for non-microstates free entropy dimension with
  applications to {$q$}-semicircular families},
        date={2004},
        ISSN={1073-7928},
     journal={Int. Math. Res. Not.},
      number={51},
       pages={2757\ndash 2772},
         url={http://dx.doi.org/10.1155/S1073792804140476},
      review={\MR{2130608 (2006e:46074)}},
}

\bib{Sn04}{article}{
      author={{\'S}niady, Piotr},
       title={Factoriality of {B}o\.zejko-{S}peicher von {N}eumann algebras},
        date={2004},
        ISSN={0010-3616},
     journal={Comm. Math. Phys.},
      volume={246},
      number={3},
       pages={561\ndash 567},
         url={http://dx.doi.org/10.1007/s00220-003-1031-1},
      review={\MR{2053944 (2005d:46130)}},
}

\bib{Sp93}{article}{
      author={Speicher, Roland},
       title={Generalized statistics of macroscopic fields},
        date={1993},
        ISSN={0377-9017},
     journal={Lett. Math. Phys.},
      volume={27},
      number={2},
       pages={97\ndash 104},
         url={http://dx.doi.org/10.1007/BF00750677},
      review={\MR{1213608 (94c:81096)}},
}

\bib{VDN}{book}{
      author={Voiculescu, D.~V.},
      author={Dykema, K.~J.},
      author={Nica, A.},
       title={Free random variables},
      series={CRM Monograph Series},
   publisher={American Mathematical Society, Providence, RI},
        date={1992},
      volume={1},
        ISBN={0-8218-6999-X},
        note={A noncommutative probability approach to free products with
  applications to random matrices, operator algebras and harmonic analysis on
  free groups},
      review={\MR{1217253 (94c:46133)}},
}

\bib{Voi02}{article}{
      author={Voiculescu, Dan},
       title={Cyclomorphy},
        date={2002},
        ISSN={1073-7928},
     journal={Int. Math. Res. Not.},
      number={6},
       pages={299\ndash 332},
         url={http://dx.doi.org/10.1155/S1073792802105046},
      review={\MR{1877005 (2003a:46092)}},
}

\bib{Voi06}{incollection}{
      author={Voiculescu, Dan},
       title={Symmetries arising from free probability theory},
        date={2006},
   booktitle={Frontiers in number theory, physics, and geometry. {I}},
   publisher={Springer, Berlin},
       pages={231\ndash 243},
         url={http://dx.doi.org/10.1007/978-3-540-31347-2_6},
      review={\MR{2261097 (2008e:46082)}},
}

\bib{Voi98}{article}{
      author={Voiculescu, Dan},
       title={The analogues of entropy and of {F}isher's information measure in
  free probability theory. {V}. {N}oncommutative {H}ilbert transforms},
        date={1998},
        ISSN={0020-9910},
     journal={Invent. Math.},
      volume={132},
      number={1},
       pages={189\ndash 227},
         url={http://dx.doi.org/10.1007/s002220050222},
      review={\MR{1618636 (99d:46087)}},
}

\end{biblist}
\end{bibdiv}
\end{document}